\documentclass[journal]{IEEEtran}
\ifCLASSINFOpdf
   \usepackage[pdftex]{graphicx}
\else
   \usepackage[dvips]{graphicx}
\fi
\usepackage{comment}

\usepackage{amsmath}
\usepackage{amssymb}
\usepackage{xspace}
\usepackage{amsthm}
\newtheorem{thm}{Theorem}
\newtheorem{fact}{Fact}
\newtheorem{lemma}{Lemma}
\newtheorem{definition}{Definition}
\newtheorem{corollary}{Corollary}
\newtheorem{proposition}{Proposition}
\newtheorem{construction}{Construction}
\newtheorem{remark}{Remark}
\newtheorem{hypothesis}{Hypothesis}

\newcommand{\term}[1]{\textsl{#1}\xspace}

\newcommand{\reflem}[1]{Lemma~\ref{lem:#1}\xspace}
\newcommand{\refthm}[1]{Theorem~\ref{thm:#1}\xspace}
\newcommand{\refdef}[1]{Definition~\ref{def:#1}\xspace}
\newcommand{\refcor}[1]{Corollary~\ref{cor:#1}\xspace}
\newcommand{\reffact}[1]{Fact~\ref{fact:#1}\xspace}

\newcommand{\refsec}[1]{Section~\ref{sec:#1}\xspace}

\newcommand{\reffig}[1]{Figure~\ref{fig:#1}\xspace}
\newcommand{\refeq}[1]{Eqn.~\eqref{eq:#1}\xspace}

\newcommand{\ga}{grid agent\xspace}

\newcommand{\transpose}{\ensuremath{T}\xspace}

\newcommand{\reals}{\ensuremath{\mathbb{R}}\xspace}
\newcommand{\powerset}[1]{\ensuremath{\mathcal{P}(#1)}\xspace}
\newcommand{\natnum}{\ensuremath{\mathbb{N}}\xspace}
\newcommand{\posint}{\ensuremath{\mathbb{N}_{>0}}\xspace}
\newcommand{\mcal}[1]{\ensuremath{\mathcal{#1}}\xspace}
\newcommand{\cl}{\textsc{commelec}\xspace}

\newcommand{\pqprof}{$PQ$ profile\xspace}
\newcommand{\belieffunc}{belief function\xspace}

\newcommand{\niek}[1]{\textsf{Niek: #1}\xspace}

\newcommand{\Preq}{\ensuremath{P^\textnormal{req}}\xspace}
\newcommand{\Pimp}{\ensuremath{P^\textnormal{imp}}\xspace}

\newcommand{\map}[1]{#1-approximation\xspace}

\renewcommand{\l}{\left}
\renewcommand{\r}{\right}

\newcommand{\conv}{\operatorname{conv}}

\newcommand{\temph}[1]{{\sl #1}}
\newcommand{\diam}{\operatorname{diam}}
\newcommand{\spn}{\operatorname{span}}
\DeclareSymbolFont{AMSb}{U}{msb}{m}{n}
\DeclareMathSymbol{\N}{\mathbin}{AMSb}{"4E}
\DeclareMathSymbol{\ZZ}{\mathbin}{AMSb}{"5A}
\DeclareMathSymbol{\Y}{\mathbin}{AMSb}{"59}
\DeclareMathSymbol{\U}{\mathbin}{AMSb}{"55}
\DeclareMathSymbol{\RR}{\mathbin}{AMSb}{"52}
\DeclareMathSymbol{\Q}{\mathbin}{AMSb}{"51}
\DeclareMathSymbol{\Prob}{\mathbin}{AMSb}{"50}
\DeclareMathSymbol{\HHH}{\mathbin}{AMSb}{"48}
\DeclareMathSymbol{\I}{\mathbin}{AMSb}{"49}
\DeclareMathSymbol{\C}{\mathbin}{AMSb}{"43}
\DeclareMathSymbol{\E}{\mathbin}{AMSb}{"45}
\DeclareMathSymbol{\A}{\mathbin}{AMSb}{"41}
\DeclareMathSymbol{\C}{\mathbin}{AMSb}{"43}
\DeclareMathSymbol{\D}{\mathbin}{AMSb}{"44}
\DeclareMathSymbol{\OO}{\mathbin}{AMSb}{"4F}
\DeclareMathSymbol{\X}{\mathbin}{AMSb}{"58}
\DeclareMathSymbol{\LL}{\mathbin}{AMSb}{"4C}

\newcommand{\Il}{\mathcal{I}}

\newcommand{\Al}{\mathcal{A}}
\newcommand{\Bl}{\mathcal{B}}

\newcommand{\Ul}{\mathcal{U}}

\newcommand{\Dl}{\mathcal{D}}

\newcommand{\HH}{\mathcal{H}}

\newcommand{\defas}{:=} 

\begin{document}
%
\title{Design of Resource Agents with Guaranteed Tracking Properties for Real-Time Control of Electrical Grids}

%
%
%

\author{
  Andrey~Bernstein, 
  Niek~J.~Bouman, 
and Jean-Yves~Le~Boudec\\
EPFL, Switzerland}
\maketitle

\begin{abstract}
  We focus on the problem of controlling electrical microgrids with little inertia in real time. We consider a centralized controller and a number of resources, where each resource is either a load, a generator, or a combination thereof, like a battery. The centralized controller periodically computes new power setpoints for the resources  based on the estimated state of the grid and an overall objective, and subject to safety constraints. Each resource is augmented with a \emph{resource agent} that a) implements the power-setpoint requests sent by the controller on the resource, and b) translates device-specific information about the resource into an abstract, device-independent representation and transmits this to the controller. 

We focus on the resource agents (including the resources that they manage) and their impact on the overall system's behavior. Intuitively, for the system to converge to the overall objective, the resource agents should be \emph{obedient} to the requests from the controller, in the sense that the actually implemented setpoint should be close to the requested setpoint, at least on average. This can be important especially when a controller that performs \emph{continuous} optimization is used (because of performance constraints) to control \emph{discrete} resources (for which the set of implementable setpoints is discrete).

In this work, we formalize this type of obedience by defining the notion of \emph{$c$-bounded accumulated-error} for some constant $c$. We then demonstrate the usefulness of our notion, by presenting theoretical results (for a simplified scenario) as well as some simulation results (for a more realistic setting) that indicate that, if all resource agents in the system have \emph{$c$-bounded accumulated-error}, the closed-loop system converges on average to the objective.
Finally, we show how to design a resource agent that provably has $c$-bounded accumulated-error for various types of resources, such as resources with uncertainty (e.g., PV panels) and resources with a discrete set of implementable setpoints (e.g., heating systems with heaters that each can either be switched on or off).

\end{abstract}


%

\section{Introduction}
\IEEEPARstart{W}{e} consider the problem of controlling a collection of electrical resources that are interconnected via an electrical grid, with respect to a (typically time-varying) objective and under certain safety constraints. These resources can be generators, loads, as well as resources that can both inject and consume power, like storage units. This problem has recently received renewed interest through the advent of renewable energy like solar power and improved battery technologies.

We focus on \emph{real-time} control, namely on sub-second time scale.
The classical approach involves a combination of both
frequency and voltage regulation using \emph{droop controllers}. With the increased penetration
of stochastic resources, distributed generation and demand response, this approach shows
severe limitations in both the optimal and feasible operation of medium and low voltage networks, as well as
in the aggregation of the network resources for upper-layer power systems. An alternative
approach is to directly control the targeted grid by defining explicit and real-time setpoints
for active/reactive power absorptions/injections defined by a solution of a specific
optimization problem. Such an approach typically assumes a bi-directional communication between the \emph{grid controller} and the different resources in the network. The communication capability enables the grid controller to have fine-grained knowledge of the system's state, which allows for a better operation of the system.

One of the major challenges in this context is
to be able to efficiently control \emph{heterogeneous resources} in real-time. The resources can have continuous or discrete nature
(e.g., heating systems consisting of a finite number of heaters that each can either be switched on or off) and/or can be highly uncertain (e.g., PV panels or residential loads). Hence, a naive approach would lead to a \emph{stochastic mixed-integer} optimization problem to be solved at the grid controller at each time step. Since the goal is real-time control, this approach is practically infeasible.

The recently introduced \cl framework for real-time control of electrical grids \cite{commelec, commelec2} inherently avoids
performing stochastic mixed-integer optimization.
The framework uses a hierarchical system of \emph{software agents}, each responsible for a single resource (loads, generators and storage devices; Resource Agents - RA) or an entire subsystem (including a grid and/or a number of resources; Grid Agents - GA).
Each resource agent advertises to its grid agent its internal state via a \emph{device-independent} protocol. In particular, the protocol requires from every RA to advertise (i) a \emph{convex} set of feasible setpoints, and (ii) an uncertainty (belief) set in setpoint implementation.
This way, the grid agent can solve a \emph{robust continuous} (rather than stochastic mixed-integer) optimization problem and send continuous setpoints to the resource agents. It is then the task of the RA to map the received setpoint to the set of actually implementable setpoints at this moment.

In this paper, we address the following two questions that arise during RA design in this context:
\begin{enumerate}
\item[(i)] How should an RA advertise the flexibility and uncertainty of the resource to a GA that works with continuous setpoints?
\item[(ii)] Given a requested setpoint from the GA (which is not necessarily implementable), which setpoint should the RA implement?
\end{enumerate}
It is important to note that the RA design can have a significant impact on the performance of the overall system.
Indeed, as it is shown in \refsec{num}, a straightforward approach in which the requested setpoint is merely projected to the set of implementable setpoints can lead to highly sub-optimal behaviour.

The key point is that the RAs should be \emph{obedient} to the GA's requests in some sense. As the first main contribution of this paper, we propose \emph{the boundedness of the accumulated error between requests and implemented setpoints} as a desired metric for obedience.
One the main motivations of this choice is that the average requested and implemented setpoints are the same in the long-run. This has an energy interpretation: the produced/consumed energy converges to the requested one. The latter can be useful, e.g., in \emph{virtual power plant}-related applications.
Further, we perform a theoretical analysis of the closed-loop system that includes a GA and a number of RAs, where each RA has bounded accumulated-error. The analysis is performed in a restricted scenario, under some simplifying assumptions. To complement this, we illustrate the performance of our method in simulation using a more realistic scenario.

The second main contribution of the paper is a framework for RA design that guarantees boundedness of the accumulated error, which covers a broad range of resources.
Our approach is conceptually simple and is inspired by a generic error-feedback technique that is known under different names in different fields;
in image processing, it is called \emph{error diffusion} (and a specific variant is called \emph{Floyd--Steinberg dithering} \cite{floyd75}), whereas in signal processing, the technique is related to sigma-delta modulation \cite{Anastassiou} and some stability properties of this technique have been analyzed in that context \cite{gray,adams}.

We briefly outline the key ideas below. Consider a sequence of requested setpoints $(x_k)$. Assume that the resource agent cannot implement $x_k$ precisely, but rather can implement any setpoint $y_k$ in a given set $\Il_k$ (not necessarily convex, possibly discrete). Let $e_k = \sum_{i = 1}^k (y_i - x_i)$ be the accumulated error. Then, roughly speaking, we answer the two questions posed above as follows. First, advertise the \emph{convex hull} of $\Il_k$ as the set of feasible setpoints. Second, when asked to implement $x_k$, actually implement $y_k = \mathrm{proj}_{\Il_k} (x_k - e_{k-1})$, where $\mathrm{proj}_{\Il_k} (\cdot)$ is a projection operator. Under certain conditions (detailed in \refsec{drr}), this method ensures boundedness of $e_k$.

Our method can also be interpreted from the perspective of PI control.
Indeed, it can be viewed as an ``I-controller'' in the following sense. Observe that given the previous accumulated error  $e_{k-1}$ and request $x_k$, we have that $y_k$ chosen above minimizes the absolute value of the new accumulated error $| e_k |  = | e_{k-1} + (y_k - x_k) | = | y_k - (x_k - e_{k-1}) |  $ over all $y_k$ in $\Il_k$.


Our results hold universally, for any sequence of requested setpoints. This allows for a separation of concerns: the designer of an RA does not depend on a particular optimization algorithm applied in the GA.

The structure of the paper is as follows. In \refsec{prelim}, we introduce notation and briefly overview the \cl framework. In \refsec{perfmetric}, we introduce the concept of the $c$-bounded accumulated-error and discuss its basic properties. In \refsec{ga_opt}, we study analytically the effect of having resource agents
with bounded accumulated-error in a system where those resource agents are controlled by a grid agent. Then, we devise a method for RA design for \emph{discrete} resource that guarantees boundedness of the accumulated error in \refsec{drr}, and for \emph{uncertain} resources in \refsec{uncertain}. In \refsec{gen}, we present a unified approach that allows to prove the results of Sections \ref{sec:drr} and \ref{sec:uncertain}, as well as to design RAs for different types of resources. In \refsec{num}, we present a numerical evaluation of the proposed methods. Finally, we close with concluding remarks in \refsec{conc}.

\section{Preliminaries}
\label{sec:prelim}
In this section, we introduce some notation and also briefly review the properties of the \cl framework that are used in the paper.

\subsection{Notation}
Throughout the paper, $\|\cdot\|$ denotes the $\ell_2$ norm. We use $\natnum$ and  $\posint$ to refer to the natural numbers including and excluding zero, respectively. For arbitrary $n\in \natnum_{>0}$, we write $[n]$ for the set $\{1,\ldots,n\}$. We use $\boldsymbol{0}$ to denote the zero vector $(0,\ldots,0)$, where the dimension should be clear from the context.
Let $d\in \natnum_{>0}$.
For an arbitrary set $\mcal{U}\in \reals^d$, $-\mcal{U}:=\{ -u \,|\, \forall u\in \mcal{U}\}$.
For arbitrary sets $\mcal{U},\mcal{V} \in \reals^d$,  $\mcal{U}+\mcal{V}$ represents the Minkowski sum of \mcal{U} and \mcal{V}, which is defined as
$\mcal{U}+\mcal{V} := \{u+v \,|\, \forall u \in \mcal{U}, \forall v\in \mcal{V} \}$.



  Fix $d\in \natnum_{>0}$.
  Let $\mcal{S} \subset \reals^d$ be an arbitrary non-empty closed set.
  Any mapping $\mathrm{proj}_{\mcal{S}}:\reals^d \rightarrow \mcal{S}$ that satisfies
\[
  \mathrm{proj}_\mcal{S}(x) =  s, \quad\text{where}\quad
s \in \arg \min_{\rho \in\mcal{S}} \|\rho-x \|.
\]
is called  a \term{projection operator} onto \mcal{S}.
%
%
%

A \term{power setpoint} is a tuple $u = (P,Q)\in \mathbb{R}^2$, where $P$ denotes real power and $Q$ denotes reactive power.
Let $\mathcal{B}(\mathbb{R}^2)$ denote the collection of all \emph{bounded} non-empty closed subsets of $\mathbb{R}^2$, and let
 $\mathcal{C}(\mathbb{R}^2)$ denote the collection of all \emph{convex} sets in $\mathcal{B}(\mathbb{R}^2)$.

 For vectors $v\in \reals^2$, as well as for sets $\mcal{V} \subseteq \reals^2$, we use superscripts $P$ and $Q$ to refer to the first and second coordinate, respectively. I.e., $v= (v^P,v^Q)$ and $\mcal{V} = \mcal{V}^P \times \mcal{V}^Q $ hold.

For any compact set \mcal{V}, we define the
\term{diameter of \mcal{V}} as
\[
\diam \mcal{V} := \max \{ \| v-w \| : v,w \in \mcal{V} \}.
\]

For any matrix $M$, we write $M \succeq 0$ to state that it is \emph{positive semi-definite}, and $M \succ 0$ to state that it is \emph{positive definite}. Also, we write $\l\| M \r\|$ to denote its \emph{induced $\ell_2$ matrix norm}.

\subsection{The COMMELEC Framework}

In the \cl framework \cite{commelec}, the RAs periodically advertise to the GA the set of (power) setpoints that they can currently implement, as well as, for each setpoint in this set, its ``cost'' (which exposes the individual objective of the resource) and the accuracy with which the setpoint can be implemented.
The GA collects those advertisements and uses them to periodically compute new setpoints, which are then sent back to the resources as requests, i.e., the RAs are supposed to implement these setpoints.
(Note that the setpoint-computation might depend on auxiliary inputs, such as the state of the electrical grid.)
Below, we present the advertised elements in more detail.

\begin{definition}
A \term{$PQ$ Profile} is a convex set $\mcal{A} \in \mathcal{C}(\mathbb{R}^2)$.
\label{def:pqprof}
\end{definition}
A \pqprof of a RA 
represents the collection of power setpoints that 
the RA is able to implement.
As explained in the Introduction, the convexity requirement on the \pqprof can be viewed as a limitation that originates from the control algorithm that is currently used in the \ga.\footnote{This particular control algorithm computes orthogonal projections onto the \pqprof, which are well-defined only if that set is convex.}

\begin{definition}
A \term{Belief Function} is a set-valued function
$$\text{BF}: \mathcal{A} \rightarrow \mathcal{B}(\mathbb{R}^2).$$ 
\end{definition}
For every setpoint $u \in \mcal{A}$, the belief function represents the uncertainty in this setpoint implementation: 
when the RA is requested to implement a setpoint $u\in \mcal{A}$, the RA 
states that the actually implemented setpoint (which could depend on external factors, for example on the weather in case of a PV)
lies in the set $\text{BF}(u)$.

\begin{definition}
A \term{Virtual Cost Function} is a continuously differentiable function $\text{CF}: \mcal{A} \rightarrow \mathbb{R}$.
\end{definition}

The virtual cost function represents the RA's aversion (corresponding to a high cost) or preference (respectively, low cost) towards a given setpoint. For example, a battery agent whose battery is fully charged will assign high 
cost to setpoints that correspond to further charging the battery.
The the adjective \emph{virtual} makes clear that we do not mean a monetary value, however, from now on we will omit this adjective and simply write ``cost function''.

As the objects defined above may change at every \cl cycle, we denote the cycle (or step) index by using subscript $k = 1, 2, \ldots$ where needed, e.g., $\text{CF}_k$, $\text{BF}_k$, etc.

At every time step $k \in \natnum_{>0}$, a resource agent receives from its grid agent a \emph{request} to implement a setpoint $u_k^\text{req} = (P_k^\text{req}, Q_k^\text{req}) \in \Al_k$. We denote
the setpoint that the resource \emph{actually implements}
by  $u_k^\text{imp} = (P_k^\text{imp}, Q_k^\text{imp}) \in \Il_k$. Here, $\Il_k \subseteq \reals^2$ denotes the set of \term{implementable setpoints} at time step $k$, namely the set describing the \term{feasibility constraints} of the resource.

Note that the implemented setpoint $u_k^\text{imp}$ need not be equal to the request $u_k^\text{req}$;
there is typically some error between them.
For example, there might be setpoints in the \pqprof that do not correspond to implementable setpoints (we will see examples of this in \refsec{drr}). Also, the setpoint that is actually implemented might depend on uncertain external factors (like the solar irradiance, in case of a PV).

\subsection{Policy for a Resource Agent}
It is convenient to formulate the decision process at a resource agent as a \emph{repeated game} between the RA and the ``environment'', which includes the effects of the decisions of the GA, Nature, and other external factors.
To that end, 
let
\begin{equation}
\HH_k = \l\{(\Al_{i}, CF_{i}, BF_i, u^{\textrm{imp}}_{i}), (u^{\textrm{req}}_{i}, \Il_{i}) \r\}_{i = 1}^k
\end{equation}
denote the \term{history} of the decision process up to (and including) time step $k$. Here, the first tuple represents an action of the RA, whereas the second tuple represents an action of the environment. A general \emph{policy} of the RA is the collection $\pi = \{\pi_k\}_{k=1}^{\infty}$ of probability distributions $\pi_k = (\pi_k^{\Al}, \pi_k^{CF}, \pi_k^{BF}, \pi_k^{\textrm{imp}})$, such that
\[
\Al_k \sim \pi_k^{\Al}\l(\cdot | \HH_{k-1} \r),\quad CF_k \sim \pi_k^{CF}\l(\cdot | \HH_{k-1} \r),
\]
\[
BF_k \sim \pi_k^{BF}\l(\cdot | \HH_{k-1} \r),\quad u^{\textrm{imp}}_k \sim \pi_k^{\textrm{imp}}\l(\cdot | \HH_{k} \r).
\]
Note that, in general, the first three elements are chosen according to the history up to, \emph{but not including}, time step $k$, while the implemented setpoint is chosen according to the history \emph{including} time step $k$.

In this context, we make the following important distinction between \emph{deterministic} and \emph{uncertain} resources.
\begin{definition} \label{def:res_classes}
We say that a resource is \term{deterministic} if, for any time step $k \in \posint$, there exists a deterministic function $G_k$ (that is known to the RA)  such that $\Il_{k} = G_k(\HH_{k-1})$.
Otherwise, we say that the resource is \term{uncertain}.
\end{definition}
In particular, an agent for a deterministic resource knows the set of implementable setpoints $\mcal{I}_k$ at the moment of advertising the \pqprof $\mcal{A}_k$, hence its policy may depend on it explicitly. On the other hand, an agent for an uncertain resource will have to \emph{predict} $\mcal{I}_k$ in order to advertise $\mcal{A}_k$.

\section{Bounded Accumulated-Error} 
\label{sec:perfmetric}

In this section, we introduce the concept of $c$-bounded accumulated-error and we moreover propose that bounded accumulated-error should be a required property of every resource agent design.
%
The motivation for our proposal 
is two-fold. First, if an RA has bounded accumulated-error, the average implemented setpoint converges to the average requested setpoint (see Proposition \ref{lem:track} below). Second, a resource that introduces a significant error when implementing a setpoint can lead to divergence and/or sub-optimal behaviour of the overall system. We show that, under the bounded accumulated-error property,
the effect on the setpoints computed by the grid agent is bounded and vanishes on average. We prove this analytically in a restricted scenario (Section \ref{sec:ga_opt})
and in simulation for a more realistic scenario (Section \ref{sec:num}).

Let
\begin{equation}
e_{ k}  = (e^{P}_{ k} , e^Q_{ k} ) \defas  \sum_{i = 1}^k \left(u^\text{imp}_i - u^\text{req}_i \right), \quad  k \in \natnum
\label{eq:accerror}
\end{equation}
denote the \emph{accumulated-error vector} at time $k$. We define the following performance metric in terms of $e_k$.
\begin{definition}
  \label{def:bounded_error}
Let $c \in \mathbb{R}$, $0 \leq c < \infty$, be given.
  We say that a resource agent has \emph{$c$-bounded accumulated-error} if 
\[
\| e_{ k} \| \leq c, \quad k \geq 1.
\]
\end{definition}
Finally, 
when we say that a resource agent has \emph{bounded accumulated-error} (without mentioning a constant), we mean that there exists some constant $c\geq 0$ 
such that the resource agent has $c$-bounded accumulated-error.


\subsection{Vanishing Average Tracking-Error}
It is easy to show that if a resource agent has bounded accumulated-error, the average implemented setpoint converges to the average requested setpoint. In fact, we next prove a stronger result that gives the rate of convergence when the average is taken over a window of a given size. To this end, for a given sequence $x = (x_k)$,  any $k = 1, 2, \ldots$, and $m \in \{0, \ldots, k-1\}$, let
\begin{equation}
\mu_{[k-m:k]}(x) \defas \frac{1}{m+1}\sum_{i=0}^m x_{k-i},
\end{equation}
denote the average of $x$ over the last $m+1$ elements. The \emph{average tracking error} between two sequences $x = (x_k)$ and $y = (y_k)$ is then defined as
\begin{equation} \label{eqn:track_err}
    \varepsilon^\textnormal{avg}_{[k-m:k]}(x,y)  \defas \Big\| \mu_{[k-m:k]}(x) - \mu_{[k-m:k]}(y)  \Big\|.
\end{equation}

\begin{proposition} \label{lem:track}
Let $\textsf{R}$ be a resource agent having $c$-bounded accumulated-error.
Then, for any $k = 1, 2, \ldots$ and $m \in \{0, \ldots, k-1\}$, the average tracking error of $\textsf{R}$ satisfies
\[
  \varepsilon^{\textnormal{avg}}_{[k-m:k]}( u^\text{req}, u^\text{imp} )
\leq \begin{cases}
  \phantom{2}(m+1)^{-1} c & \text{if $m=k-1$,}\\
 2 (m+1)^{-1} c & \text{otherwise.}\\
\end{cases}
\]
In particular, for the case $m = 0$, we obtain an upper bound on the instantaneous error
\[
\big\| u_k^\text{imp} - u_k^\text{req} \big\| \leq 2c.
\]
\end{proposition}
\begin{proof}
\begin{align*}
  \varepsilon^\textnormal{avg}&_{[k-m:k]}(  u^\text{imp}, u^\text{req} ) =
  \Big\| \mu_{[k-m:k]}(u^\text{req}) - \mu_{[k-m:k]}(u^\text{imp})  \Big\|
  \\
  &= \frac{1}{m+1} \Big\|\sum_{i=0}^m \big(u_{k-i}^\textnormal{imp} -u_{k-i}^\textnormal{req}\big)\Big\|\\
  &=  \frac{1}{m+1} \Big\| \sum_{i=1}^k \big(u_{i}^\textnormal{imp}\!-\!u_{i}^\textnormal{req}\big)
-\hspace{-.6em}  \sum_{j=1}^{k-m-1} \hspace{-.7em}\big(u_{j}^\textnormal{imp}\!-\!u_{j}^\textnormal{req}\big)\Big\|\\
&= \frac{1}{m+1} \|e_k - e_{k-m-1} \|
\leq  \frac{1}{m+1} (\|e_k\| + \|e_{k-m-1} \|)
\end{align*}
We obtain the statement by using the bound on the error from \refdef{bounded_error} twice. For the case $m=k-1$, we use the fact that $e_0=(0,0)$ (by definition of $e_k$) which gives us the improved bound.
\end{proof}

\subsection{Example: Resource with Delay}
Here, we would like to give some simple examples of resources that have bounded accumulated-error; in Sections \ref{sec:drr} and \ref{sec:uncertain}, we study more elaborate examples. A trivial example is an \emph{ideal} device, for which $u^\text{imp}_k = u^\text{req}_k$ for every $k$, which immediately implies that $e_k = e_0$ for every $k$.

Another example of a resource with bounded accumulated error by construction is a device that has a \emph{delay} of $\tau \in \posint$ time steps when implementing a setpoint.
For simplicity, consider a resource with a fixed \emph{convex} set of implementable setpoints $\Il \subseteq \reals^2$. Namely, the set $\Il$ represents the feasibility constraint of the resource. When asked at time step $k$ to implement a setpoint $u^\text{req}_{k} \in \Il$, $u^\text{req}_{k} \neq u^\text{imp}_{k-1}$, it will only implement it after $\tau$ timesteps; $u^\text{imp}_{k + \tau} = u^\text{req}_{k}$.
During the ``transition period'', the implemented setpoint remains unchanged, i.e.,  $u^\text{imp}_{k + i} = u^\text{imp}_{k -1}$ for all $i \in \{0,\ldots, \tau - 1 \}$.
At timesteps $k +1$ up to and including $k +\tau - 1$, the \pqprof of the resource is a singleton set that corresponds to the \emph{implemented} setpoint at the previous timestep (i.e., equal to $\{u^\text{imp}_{k - 1}\}$);
while at timestep $k + \tau$, the \pqprof is set back to $\Il$. Note that
no error is accumulated at timesteps where the \pqprof is a singleton set. The idea is illustrated in Figure \ref{fig:delayed}.

\begin{figure}[h!]
\centering
\includegraphics[width=\columnwidth]{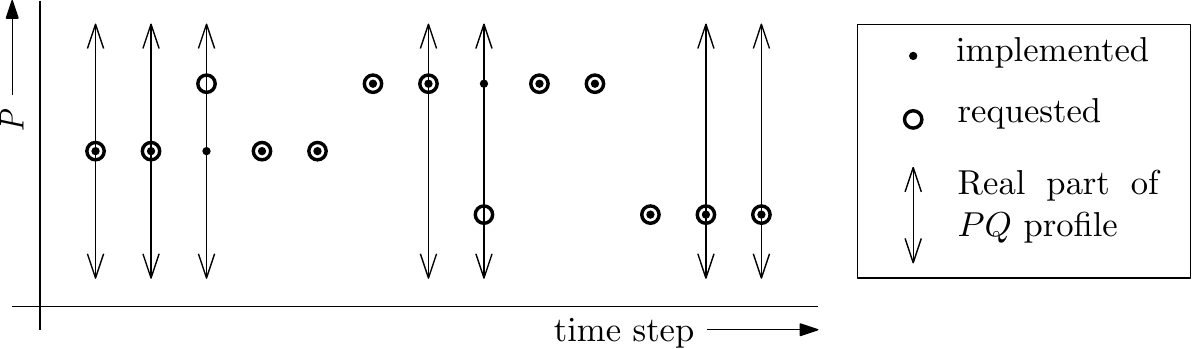}
\caption{Illustration of operation of a real-power resource with delay $\tau = 3$. \label{fig:delayed}}
\end{figure}

We show below that this resource agent has $c$-bounded accumulated error with $c = \diam \Il$.
Fix $k \in \posint$. Let $\{i_{\ell}\}_{\ell = 1}^L$, $1 < i_1 < i_2 < ... < i_L \leq k$,
be the set of time indices such that the request differs from the implemented setpoint, namely
\begin{align*}
u^{\textnormal{req}}_{i} &= u^{\textnormal{imp}}_{i}, \quad \text{for every } i\in [k] \setminus \{i_{\ell}\}_{\ell = 1}^L \\ 
u^{\textnormal{req}}_{i_{\ell}} & \neq u^{\textnormal{imp}}_{i_{\ell}}, \quad \text{for every } \ell\in [L].
\end{align*}
Observe that because of the delay, we have in particular that $u^{\textnormal{imp}}_{i_{\ell}} = u^{\textnormal{imp}}_{i_{\ell} -1}$. When the request differs from the implemented setpoint, the $PQ$ profile is ``locked'' on the singleton $\{u^{\textnormal{imp}}_{i_{\ell} - 1}\}$ during the transition period, and as a result $u^{\textnormal{req}}_{i_{\ell}+i} = u^{\textnormal{imp}}_{i_{\ell}+ i}$, $i = 1, ..., \tau$. At the end of the transition period, the implemented setpoint equals to the request right before the transition started, namely $u^{\textnormal{imp}}_{i_{\ell} + \tau} = u^{\textnormal{req}}_{i_{\ell}}$. Also, by the definition of the time indices $\{i_{\ell}\}_{\ell = 1}^L$, it holds that $u^{\textnormal{imp}}_{i_{\ell} + \tau} = u^{\textnormal{imp}}_{i_{\ell+1}}$. It is then clear that
\begin{equation} \label{eqn:telescope}
u^{\textnormal{imp}}_{i_{\ell + 1}} = u^{\textnormal{req}}_{i_{\ell}}, \quad \ell \in [L].
\end{equation}
Thus, the accumulated error in \refeq{accerror} becomes
\[
  e_{k}  = \sum_{i = 1}^{k} \l( u^\textnormal{imp}_{i} - u^\textnormal{req}_{i} \r)
   = \sum_{\ell = 1}^{L} \l( u^\textnormal{imp}_{i_{\ell}} - u^\textnormal{req}_{i_{\ell}} \r)
   = u^\textnormal{imp}_{i_{1}} - u^\textnormal{req}_{i_{L}},
\]
where the last equality follows by \eqref{eqn:telescope}.
Hence,
\[
\| e_k \| \leq \|u^\textnormal{imp}_{i_{1}} - u^\textnormal{req}_{i_{L}}\| \leq  \diam \Il.
\]

\section{Convergence Analysis of the Closed-Loop System} \label{sec:ga_opt}
In this section, we study analytically the effect of having resource agents with bounded accumulated-error in a system where those resource agents are controlled by a grid agent.

For a restricted scenario (a set of conditions that we state precisely further below), we show that the control algorithm used in the grid agent converges to the optimum (minimum) value of the objective function.






\subsection{Control Algorithm: Projected Gradient Descent}
Consider a grid agent that has $n$ follower agents and possibly a leader agent. Let $x_k = (u^{\textrm{req}}_{k, 1}, \ldots, u^{\textrm{req}}_{k, n}) \in \reals^{2n}$ denote the vector of requested setpoints computed by the grid agent at time step $k$. Similarly, let $y_k = (u^{\textrm{imp}}_{k, 1}, \ldots, u^{\textrm{imp}}_{k, n}) \in \reals^{2n}$ denote the vector of actually implemented setpoints (implemented by the resources) at time step $k$.
Let $\hat{y}_k= ((P_1,Q_1),\ldots,(P_n,Q_n)) \in \reals^{2n}$ denote the vector of estimated bus powers of the $n$ buses in the grid to which the followers are connected, where we suppose that each resource is connected to a distinct bus, and that no other loads or generators are connected to those buses.
\reffig{closedloop} shows a block diagram of this setup.

\begin{figure}
\centering
\includegraphics[width=\columnwidth]{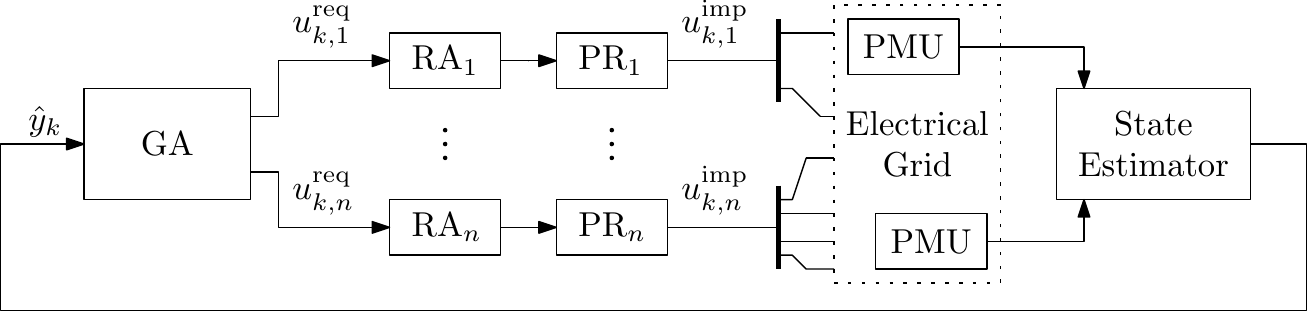}
\caption{Setup involving a grid agent (GA), resource agents (RA$_i$), the physical resources themselves (PR$_i$), the electrical grid including phasor measurement units (PMUs) and the state estimator.}
\label{fig:closedloop}
\end{figure}





As described in \cite{commelec}, the grid agent computes setpoints using a continuous gradient steering algorithm
\begin{equation} \label{eqn:ga_algo}
  x_{k+1} = \mathrm{proj}_{\Ul} \l(\hat{y}_k - \alpha \nabla J_k(\hat{y}_k) \r),
\end{equation}
where $\Ul$ is the set of \emph{admissible} setpoints as defined in \cite{commelec}, $J_k$ is the current objective function that includes the weighted sum of the followers cost functions as well as the penalty terms related to the grid quality of supply and deviation from the request from the leader (the parent grid agent), and $\alpha$ is the gradient step size. Observe that since the steering is performed from the estimated grid state $\hat{y}_k$, 
without further knowledge on the relation between $x_k$ and $\hat{y}_k$, no guarantees can be provided for this algorithm. However, 
if the resource agents are obedient in the sense of Definition \ref{def:bounded_error}, we can show that under certain conditions
the grid's state 
converges to the optimal state.

\subsection{A Restricted Scenario} 
We perform the analysis of \eqref{eqn:ga_algo} under the following set of conditions.

\begin{hypothesis} \label{asm:simpl}
Consider the setting where:
\begin{enumerate}
  \item[(i)] The state is estimated perfectly, i.e., $\hat{y}_k = y_k$, hence the gradient steering algorithm simplifies to:
\[
x_{k+1} = \mathrm{proj}_{\Ul} \l(y_k - \alpha \nabla J_k(y_k) \r).
\]
\item[(ii)] The projection in \eqref{eqn:ga_algo} is not ``active'', i.e.,
\[
y_k - \alpha \nabla J_k(y_k) \in \Ul \quad \text{for all } k.
\]
\item[(iii)] The objective function $J_k$ is fixed, namely $J_k \equiv J$.
\item[(iv)] The objective function $J$ is quadratic and convex, namely
\[
J(x) = x^{\transpose} \Gamma x + \gamma x + a
\]
for some positive definite and symmetric matrix $\Gamma \succ 0 \in \reals^{2n \times 2n}$, vector $\gamma \in \reals^{2n}$, and $a \in \reals$.
\end{enumerate}
\end{hypothesis}
Condition (i) of Hypothesis \ref{asm:simpl} expresses that the state of the grid is solely determined by the setpoints implemented by the followers, and that the grid agent knows this state exactly. We leave the analysis of the more realistic ``noisy case'' for further research. 
Condition (ii)  corresponds to
an evolution of the system in which
all operating points of the grid lie well inside the feasibility region defined by quality-of-supply constraints and the $PQ$ profiles of the different resources. We cannot argue that this condition is always satisfied in practice --- on the contrary, a main feature of the \cl framework is that it can cope with grids ``under stress''.
What we can argue, however, is that in such stressed-grid situations, \emph{safety}, rather than
optimal performance (with respect to $J_k$), is the primary concern. 
Safety of the grid is guaranteed when the vector of setpoints lies inside the 
admissible set $\Ul$, as discussed in \cite{commelec}. If, on the other hand, the grid is not stressed, then we are primarily interested in achieving (close to) optimal performance, and for this case Condition (ii) will be satisfied. 


 Condition (iii) merely requires that the rate of change of the advertisement messages is much slower than \cl's temporal resolution. Finally, Condition (iv) can be satisfied by requiring the resource agents to send only quadratic cost functions, and by linearizing the penalty term that is related to the grid QoS. In particular, this condition is satisfied in the setup considered in \cite{commelec}.

Regardless of whether Hypothesis \ref{asm:simpl} is valid in realistic scenarios,
our result may be viewed as a potential stepping stone to a more general result,
 and as a sanity check:
suppose instead that we had found an impossibility result in this simplified scenario,
then this would already have ruled out the possibility of obtaining a more general convergence result.

\subsection{Convergence Result}
For the restricted scenario 
described above, we prove the following result.

\begin{thm} \label{thm:main}
  Consider a system
consisting of:
\begin{itemize}
\item a grid agent, with control law
\eqref{eqn:ga_algo} under the conditions stated in Hypothesis \ref{asm:simpl} and with step size $\alpha \leq 1/\rho(\Gamma)$, where $\rho(\Gamma)$ is the spectral radius of $\Gamma$;
\item $n \in \natnum_{>0}$ follower resource agents, where for every $i \in [n]$, resource agent $i$ has $c_i$-bounded accumulated-error. Let $c:= \begin{bmatrix} c_1 & \ldots & c_n \end{bmatrix} \in \reals^n$.
\end{itemize}
Then, it holds that 
\[
  \lim_{k \rightarrow \infty} \frac 1 k \sum_{i = 1}^k y_k =  x^*,\quad \text{where }
x^* \defas \arg \min_{x} J(x) = -\Gamma^{-1} \gamma.
\]
%
Moreover, the rate of convergence is given by
\[
\l \| \frac 1 k \sum_{i = 1}^k y_k - x^* \r\| \leq  \frac{\l(1 - r^k \r)}{k (1 - r)} \l( \| \psi \| +  \l\|\Gamma^{-1}\r\| \l\|\gamma\r\| r^2 + \l\|c\r\| \r)
\]
where  $r \defas \rho \l ( I -\alpha  \Gamma \r) < 1$ and $\psi \defas x_1 + 2 \alpha \gamma -  \alpha^2 \Gamma \gamma$.
\end{thm}

We prove this theorem using the following two lemmas.

\begin{lemma} \label{lem:yk}
Under the conditions from Hypothesis \ref{asm:simpl}, the implemented setpoint $y_k$ satisfies the following recursion
\begin{equation} \label{eqn:y_rec}
  y_{k+1}  = (I - \alpha\Gamma) y_k - \alpha \gamma + \epsilon_{k+1}, \quad k\in \natnum, k \geq 1,
\end{equation}
where for every $k \in \natnum$
\[
  \epsilon_k \defas \begin{cases} y_k - x_k & \text{if } k \in \natnum, k \geq 1 \\
                     \boldsymbol{0} & \text{if } k = 0 \end{cases}
\]
is the instantaneous error. Moreover, $y_k$ is given explicitly by
\begin{align}
  y_k& = (I-\alpha \Gamma)^{k-1 } \psi + \sum_{i=0}^k (I - \alpha \Gamma)^{k-i} \epsilon_i \notag \\
     & \quad \quad \quad \quad \quad \quad \quad - \sum_{j=0}^k (I - \alpha \Gamma)^{j} (\alpha \gamma) \label{eqn:y_1} \\
     & = (I-\alpha \Gamma)^{k-1 } \psi + \sum_{i=0}^k (I - \alpha \Gamma)^{k-i} \epsilon_i \notag \\
     & \quad \quad \quad \quad \quad \quad \quad - \Gamma^{-1} (I - (I-\alpha \Gamma)^{k+1}) \gamma \label{eqn:y_2},
\end{align}
for every $k \in \natnum, k \geq 1$,
where $\psi \defas x_1 + 2 \alpha \gamma -  \alpha^2 \Gamma \gamma$ is a constant vector that depends on the initial condition, i.e., the first request $x_1$.
\end{lemma}

\begin{proof}
First observe that under Hypothesis \ref{asm:simpl}, the algorithm \eqref{eqn:ga_algo} becomes
\[
x_{k+1} = y_k - \alpha\, (\Gamma y_k + \gamma) = (I - \alpha\Gamma) y_k - \alpha \gamma.
\]
Also, by definition, $x_{k+1} = y_{k+1} - \epsilon_{k+1}$. Hence, we obtain recursion \eqref{eqn:y_rec}.
Next, we show that \eqref{eqn:y_1} satisfies \eqref{eqn:y_rec}. Indeed, plugging \eqref{eqn:y_1} in the right-hand-side of \eqref{eqn:y_rec} yields
\begin{align*}
& (I-\alpha \Gamma)^{k } \psi + \sum_{i=0}^k (I - \alpha \Gamma)^{k-i + 1} \epsilon_i - \sum_{j=0}^k (I - \alpha \Gamma)^{j+1} (\alpha \gamma) \\
& \quad \quad \quad \quad \quad \quad \quad - \alpha \gamma + \epsilon_{k+1} \\
& = (I-\alpha \Gamma)^{k } \psi + \sum_{i=0}^{k+1} (I - \alpha \Gamma)^{(k + 1) - i} \epsilon_i - \sum_{j=0}^{k+1} (I - \alpha \Gamma)^{j} (\alpha \gamma) \\
& = y_{k+1}
\end{align*}
as required.
We obtain the expression for $\psi$ by evaluating our explicit expression for $y_k$ at $k=1$, i.e.,
\[
  y_1 = \psi + \epsilon_1 + (I-\alpha \Gamma ) \epsilon_0 - \alpha \gamma - (I-\alpha \Gamma )(\alpha \gamma),
\]
and then solving for $\psi$ and using that $\epsilon_0 = \boldsymbol{0}$ and that $y_1 - \epsilon_1 = x_1$.

Finally, \eqref{eqn:y_2} follows by using the formula for geometric series of matrices. Namely, for any square matrix $A$ we have that
\[
\sum_{i = 0}^k A^i = (I - A)^{-1} (I - A^k)
\]
provided that $I - A$ is invertible. In our case, $A = I - \alpha \Gamma$, and $I - A = \alpha \Gamma$ is invertible by Hypothesis \ref{asm:simpl} (iv).
\end{proof}

\begin{lemma} \label{lem:eigen}
  We have that $(I-\alpha \Gamma) \succeq 0$ and $\rho(I-\alpha \Gamma) < 1$ if and only if $\Gamma \succ 0$ and $\rho(\Gamma)\leq 1/\alpha$.
\end{lemma}

\begin{proof}
Let $\lambda_\Gamma$ be an eigenvalue of $\Gamma$. We have that
  \begin{align*}
    0 &< \lambda_\Gamma \leq \frac{1}{\alpha}  \\
    \Leftrightarrow 0 &> -\alpha \lambda_\Gamma \geq -1  \\
    \Leftrightarrow 1 &> 1 -\alpha \lambda_\Gamma \geq 0.
  \end{align*}
  However, $1 -\alpha \lambda_\Gamma$ is an eigenvalue of $I-\alpha \Gamma$ if and only if $\lambda_\Gamma$ is an eigenvalue of $\Gamma$, which completes the proof.
\end{proof}

\begin{proof}[Proof of Theorem \ref{thm:main}]
We use Lemma \ref{lem:yk}, and take the average of $y_1, \ldots, y_k$:
\begin{align*}
  \frac{1}{k}\sum_{i = 1}^k y_i & = \frac{1}{k}\sum_{i = 1}^k \Big[ (I-\alpha  \Gamma )^{i-1} c_1  + \sum_{j=0}^i (I - \alpha \Gamma)^{i-j} \epsilon_i \\
  & \quad \quad \quad \quad \quad- \Gamma^{-1} (I - (I-\alpha \Gamma)^{i+1}) \gamma \Big] \\
  & = -\Gamma^{-1}\gamma + \l(\frac{1}{k}\sum_{i = 1}^k  (I -\alpha  \Gamma )^{i-1}\r)c_1  \\
  & \quad + \Gamma^{-1} \l(\frac{1}{k}\sum_{i = 1}^k  (I-\alpha \Gamma)^{i+1} \r) \gamma \\
  & \quad + \frac{1}{k}\sum_{i = 1}^k \sum_{j=0}^i (I - \alpha \Gamma)^{i-j} \epsilon_i.
\end{align*}
Note that $x^* = -\Gamma^{-1}\gamma$ and set $r \defas \rho \l ( I -\alpha  \Gamma \r)$. We have that
\begin{align}
  &\l \| \frac{1}{k}\sum_{i = 1}^k y_i - x^* \r\|  \notag \\
  & \leq  \frac{\l \|c_1\r\|}{k}\sum_{i = 1}^k  \l \|(I -\alpha  \Gamma)^{i-1}\r \|   +  \frac{\l\|\Gamma^{-1}\r\| \l\|\gamma\r\|}{k}\sum_{i = 1}^k  \l\|(I-\alpha \Gamma)^{i+1}\r\| \notag  \\
  & \quad + \l\|\frac{1}{k}\sum_{i = 1}^k \sum_{j=0}^i (I - \alpha \Gamma)^{i-j} \epsilon_i \r\| \notag\\
  & = \frac{\l\|c_1\r\|}{k}\sum_{i = 1}^k  \rho \l ( (I -\alpha  \Gamma)^{i-1}\r ) + \frac{\l\|\Gamma^{-1}\r\| \l\|\gamma\r\|}{k}\sum_{i = 1}^k  \rho \l ((I-\alpha \Gamma)^{i+1}\r )  \notag \\
  & \quad+ \l\|\frac{1}{k}\sum_{i = 1}^k \sum_{j=0}^i (I - \alpha \Gamma)^{i-j} \epsilon_i \r\| \notag\\
  & = \frac{ \l \|c_1\r\|}{k}\sum_{i = 1}^k  r^{i-1}   + \frac{\l\|\Gamma^{-1}\r\|  \l\|\gamma\r\|}{k}\sum_{i = 1}^k  r^{i+1}  \notag \\
  & \quad+ \l\|\frac{1}{k}\sum_{i = 1}^k \sum_{j=0}^i (I - \alpha \Gamma)^{i-j} \epsilon_i \r\| \notag \\
  & = \frac{\l \|c_1\r\| \l(1 - r^k \r)}{k (1 - r)} + \frac {\l\|\Gamma^{-1}\r\| \l\|\gamma\r\| r^2 \l(1 - r^{k}\r)} {k (1 - r)} \notag \\
  & \quad+ \l\|\frac{1}{k}\sum_{i = 1}^k \sum_{j=0}^i (I - \alpha \Gamma)^{i-j} \epsilon_i \r\|. \label{eqn:bound}
\end{align}
In the above sequence, the first equality follows by the fact that for any symmetric matrix $A$, $\| A \|_2 = \rho(A)$; and the second equality holds by the fact that $\rho(A^i) = \l(\rho(A)\r)^i$. (The latter holds for an arbitrary $A$.) Observe that by the hypothesis of this theorem and Lemma \ref{lem:eigen}, $r < 1$. Thus, the first two terms in \eqref{eqn:bound} converge to zero. As for the third term, we have
\begin{align*}
  & \bigg \|\frac{1}{k}\sum_{i = 1}^k \sum_{j=0}^i (I - \alpha \Gamma)^{i-j} \epsilon_i \bigg\| = \\
  &\frac{1}{k} \big\| \l[ (I - \alpha \Gamma) \epsilon_0 + \epsilon_1\r] + \l[(I - \alpha \Gamma)^2 \epsilon_0 + (I - \alpha \Gamma) \epsilon_1 + \epsilon\r] + \\
  &\ldots  + \l[(I - \alpha \Gamma)^k \epsilon_0 + \ldots + \epsilon_k\r]\big\| \\
& = \frac{1}{k} \Bigg\|\sum_{i=1}^k \epsilon_i\ +  (I - \alpha \Gamma) \sum_{i=0}^{k-1} \epsilon_i  + (I - \alpha \Gamma)^2 \sum_{i=0}^{k-2} \epsilon_i \\
& \quad \quad \quad + \ldots + (I - \alpha \Gamma)^k \epsilon_0\Bigg\|\\
& \leq \frac{1}{k} \l( \l\| \sum_{i=1}^k \epsilon_i \r\|  + r \l\| \sum_{i=0}^{k-1} \epsilon_i  \r\| + r^2 \l\| \sum_{i=0}^{k-2} \epsilon_i  \r\| + \ldots + r^k \l\| \epsilon_0 \r\| \r) \\
& \leq \frac{\l\|c\r\|}{k} \sum_{i = 0}^{k-1} r^i = \frac{\l\|c\r\| \l(1 - r^k \r)}{k (1 - r)},
\end{align*}
where the second inequality follows by the premise that all resource agents have bounded accumulated-error, and the fact that $\epsilon_0 = \boldsymbol{0}$. Combining this with \eqref{eqn:bound} completes the proof of the Theorem.
\end{proof}

\section{Bounded Accumulated-Error: Discrete (Real-Power) Resources}
\label{sec:drr}
In this section, we show how to achieve boundedness of the accumulated error for
resources that can only implement power setpoints from a \emph{discrete} set.
As an application, we consider a heating system consisting of a finite number of heaters that each can either be switched on or off (see \refsec{building} below).
We focus here on \emph{deterministic} resources in the sense of \refdef{res_classes}.
We defer the treatment of uncertain resources to \refsec{uncertain}.

\subsection{A General Construction}
To keep the exposition simple, and because it suffices for our concrete examples presented in \refsec{building}, we restrict here to a scenario with a resource that only produces/consumes \emph{real power}.

\begin{definition} 
\label{def:max_step}
For any \term{finite} non-empty set $\mcal{S} \subset \reals$, whose elements we label as $s_1 < s_2 < \ldots < s_{|\mcal{S}|}$, we define the \term{maximum stepsize of \mcal{S}} as
  \[
    \Delta_\mcal{S} \defas
      \begin{cases}
        0 & \text{if $|\mcal{S}|=1$}\\
         \max_{i \in [|\mcal{S}|-1]} s_{i+1} - s_i & \text{if $|\mcal{S}|>1$}.
  \end{cases}
  \]
\end{definition}


\begin{definition} 
\label{def:max_step_coll}
  For every finite non-empty collection of sets $\mathbb{S}\defas \{ \mcal{S}_i \}_{i \in [n]}$ where $n\in \natnum$ and
where $\mcal{S}_i  \subset \reals$ is a finite non-empty set for all $i\in [n]$, we define the maximum step size of $\mathbb{S}$ as
\[
  \Delta_\mathbb{S}\defas \max_{i \in [n]} \Delta_{\mcal{S}_i}.
\]
\label{def:collstep}
\end{definition}
The following theorem states our main result for discrete real-power resources. Its proof is deferred to \refsec{gen}.

\begin{thm}
  Let $\mathbb{S}$ be a finite non-empty collection of sets $\mathbb{S}\defas \{ \mcal{S}_i \}_{i \in [n]}$ where $n\in \posint$ and $\mcal{S}_i  \subset \reals$ is a finite non-empty set for all $i\in [n]$. Let $\Delta_{\mathbb{S}}$ denote the maximum stepsize of $\mathbb{S}$.
  Let $(i_k)_{k \in \posint}  $ be any sequence with $i_k \in [n]$, and let
  $P^\text{req}_k \in \conv (\mcal{S}_{i_k})$ for every $k\in \posint$. Let $e^P_k \in \reals$ be the accumulated  error as defined in \refeq{accerror} for every $k\in\natnum$.
Then, if a resource agent implements
\begin{equation}
P^{\textrm{imp}}_k = \mathrm{proj}_{\mcal{S}_{i_k}}(P^\text{req}_k - e^P_{k-1})
\label{eq:discreterealrule}
\end{equation}
for every $k\in \posint$,  
it has $\tfrac12\Delta_{\mathbb{S}}$-bounded accumulated-error.
\label{thm:discreteres}
\end{thm}

For a discrete-resource agent that implements setpoints according to \refthm{discreteres},
it follows immediately from Proposition \ref{lem:track} that
\[
|P^{\textrm{imp}}_k - P^\text{req}_k | \leq \Delta_\mathbb{S}.
\]
For a special case that we define below, we have a slightly stronger result 
that proves optimality, in some sense, of our construction.

\subsection{Accuracy of Setpoint Implementation}


\begin{definition}
  Let $\mcal{S} \subset \reals$ be as in Definition \ref{def:max_step}.
    We say that \mcal{S} is \term{uniform} if
  $\Delta_{\mcal{S}}= s_{i+1} - s_i$ for all $i\in [|\mcal{S}|-1  ]$. Similarly, we say that a collection $\mathbb{S}\defas \{ \mcal{S}_i \}_{i \in [n]}$ is \term{uniform} if
  $\mcal{S}_i$ is uniform and $\Delta_{\mcal{S}_i}=\Delta_{\mathbb{S}}$ for all $i\in [n]$.
\end{definition}

In addition, we require the following definition of a special projection operator. 
\begin{definition} 
  Fix $d\in \natnum_{>0}$.
  For any non-empty closed set $\mcal{S} \subset \reals^d$ we define
\[
  \mathrm{proj}^y_\mcal{S}(x) := \arg\min_{\rho
  \in \mcal{T}(\mcal{S},x)} \| \rho - y \|  \qquad x,y\in \reals^d
\]
where
\[
  \mcal{T}(\mcal{S},x):= \arg\min \{ \| \sigma - x  \| : \sigma \in \mcal{S} \}
\]
%
%
%
%
%
\label{def:specialproj}
\end{definition}

Informally speaking, this projection operator is such that, in case the cardinality of $\mcal{T}(\mcal{S},x)$ 
is strictly larger than one
(which happens when the argument $x$ has exactly the same distance to multiple points in \mcal{S}),
it chooses the point in $\mcal{T}(\mcal{S},x)$ that is closest to some given point $y$.

\begin{thm} \label{thm:accuracy}
Consider the setting of \refthm{discreteres}. Assume in addition that: (i) $\mathbb{S}$ is uniform and $\Delta_{\mathbb{S}}>0$, and (ii)
a resource agent implements $$P^{\textrm{imp}}_k =
\mathrm{proj}^{P^\text{req}_k}_{\mcal{S}_{i_k}}(P^\text{req}_k - e^P_{k-1})$$ for every $k\in \posint$, where the projection operator is given by Definition \ref{def:specialproj}. 
Then, the RA has $\tfrac12\Delta_{\mathbb{S}}$-bounded accumulated-error, and, additionally, the following properties hold:
\begin{enumerate}
  \item [(i)] $|  P^{\textrm{imp}}_k  - P^{\textrm{req}}_k | < \Delta_{\mathbb{S}}$ holds for every $k$ (with strict inequality), and this is the best possible bound in the following sense: For any algorithm that achieves the bounded accumulated-error property and any $\varepsilon > 0$, there exists a sequence $(P^\text{req}_k)$ and a sequence $(i_k)$ such that $|P^{\textrm{imp}}_{k_{\circ}}  - P^{\textrm{req}}_{k_{\circ}} | \geq \Delta_{\mathbb{S}} - \varepsilon$ for some $k_\circ$.
  \item[(ii)] If $P^{\textrm{req}}_{k_\circ} \in \mcal{S}_{i_{k_\circ}}$ holds for some $k_\circ$ (in words: if the request is implementable itself), then
  $P^{\textrm{imp}}_{k_\circ} = P^{\textrm{req}}_{k_\circ}$, regardless of the value of the accumulated error $e_{k_{\circ}-1}$.
\end{enumerate}

\end{thm}

\begin{proof}
Observe that by Theorem \ref{thm:discreteres}, $\l\| e_{k-1} \r\| \leq \tfrac12 \Delta_{\mathbb{S}}$ for every $k$. It is clear that if $\l\| e_{k-1} \r\| < \tfrac12 \Delta_{\mathbb{S}}$,
\[
\l|P^{\textrm{req}}_k - P^{\textrm{imp}}_k \r| = \l|P^{\textrm{req}}_k - \mathrm{proj}_{\mcal{S}_{i_k}}(P^\text{req}_k - e^P_{k-1}) \r| < \Delta_{\mathbb{S}}
\]
for \emph{any} projection operator. The boundary case is when $\l\| e_{k-1} \r\| = \tfrac12 \Delta_{\mathbb{S}}$ and $P^\text{req}_k \in \mcal{S}_{i_k}$. In such a case, as the ties are broken towards $P^\text{req}_k$ when using the operator of Definition \ref{def:specialproj}, it holds that $\big|P^{\textrm{req}}_k - P^{\textrm{imp}}_k \big| = 0$. Hence, $\big|P^{\textrm{req}}_k - P^{\textrm{imp}}_k \big|  < \Delta_{\mathbb{S}}$ for any $k$.

To conclude the proof of property (i), consider any algorithm that achieves the bounded accumulated-error property, and set $\varepsilon > 0$.
Let $i^*\in [n]$ be arbitrary and let $i_k:=i^*$ for every $k\in\posint$. Let $P_1, P_2 \in \mcal{S}_{i^*}$ such that $P_2 = P_1 + \Delta_{\mathbb{S}}$. (As $\Delta_{\mathbb{S}} > 0$, such choice is guaranteed to exist.)
Let $P^{\textrm{req}}_k := P_1 + \varepsilon$ for all $k\in \posint$. 
By the assumption on the algorithm, we have that
\[
\lim_{K \rightarrow \infty} \frac 1 K \sum_{k = 1}^K P^{\textrm{imp}}_k = P_1 +  \varepsilon.
\]
Hence, there exists $k_\circ$ such that $P^{\textrm{imp}}_{k_\circ} \geq P_2$, and $P^{\textrm{imp}}_{k_\circ} - P^{\textrm{req}}_{k_\circ} \geq P_2 - P_1 - \varepsilon =  \Delta_{\mathbb{S}} - \varepsilon$. 

Property (ii) follows then directly from property (i):  If $P^{\textrm{req}}_{k_\circ} \in \mcal{S}_{i_{k_\circ}}$ holds for some $k_\circ$, and $|  P^{\textrm{imp}}_{k_\circ}  - P^{\textrm{req}}_{k_\circ} | < \Delta_{\mathbb{S}}$, then by definition of the stepsize of $\mcal{S}_{i_k}$ in the uniform case,  we have that $P^{\textrm{imp}}_{k_\circ} = P^{\textrm{req}}_{k_\circ}$.
\end{proof}

\subsection{Example: Resource Agent for Heating a Building} \label{sec:building}

\newcommand{\Ttgt}{\ensuremath{T_\textnormal{target}}\xspace}
\newcommand{\Tmin}{\ensuremath{T_\textnormal{min}}\xspace}
\newcommand{\Tmax}{\ensuremath{T_\textnormal{max}}\xspace}
\newcommand{\Pheat}{\ensuremath{P_\textnormal{heat}}\xspace}

In this section, we present a concrete resource-agent example: we will design a resource agent for managing the temperature in a building with several rooms.
The reason for showing this example is twofold.
First, we wish to give a concrete example of a resource agent that
controls a load that can only
implement power setpoints from a discrete set. Second, the resource-agent design shows a concrete usage example of the \cl framework,
and might serve as a basis for an actual resource-agent implementation.

The heating system's objective is 
to keep the rooms' temperatures within a certain range.
For rooms whose temperature lies in that range, there is some freedom in the choice of the control actions related to  those rooms.
The resource agent's job is to monitor the building and spot such degrees of freedom, and expose them to the grid agent, which can then exploit those for performing \emph{Demand Response}.

Our example is inspired by \cite{costanzo}, which also considers the problem of controlling the temperature in the rooms of a building using multiple heaters.
We address two issues that were not addressed in \cite{costanzo}:
\begin{enumerate}
\item We show 
  that by rounding requested setpoints into implementable setpoints using \eqref{eq:asm_imp} (where $F_k$ is a suitably defined quantizer) 
  we obtain a resource agent with bounded accumulated-error.

%
%
%
%
%
\item We prevent the heaters from switching on and off with the same frequency as \cl's control frequency, which is crucial in an actual implementation. 
\end{enumerate}

\subsubsection{Simple Case: a Single Heater}
For simplicity, we first analyze a scenario with only one heater.
The main aspects of our proposed design (as mentioned above) are in fact independent of the number of heaters, and we think that those aspects are more easily understood in this simple case.
We will generalize our example to an arbitrary number of heaters in \refsec{moreheaters}.


\paragraph{Model and Intended Behavior}
We model the heater as a purely resistive load (it does not consume reactive power)
that can be either active (``on'') or inactive (``off''). It consumes $P_\text{heat}>0$ Watts while being active, and zero Watts while being inactive.

From the perspective of the resource agent, 
the heater has a \emph{state} that consists of two binary variables:
$s_k \in \{0,1\}$, which corresponds to whether the heater is on ($s_k=1$) or off ($s_k=0$), and
$\ell_k \in \{0,1\}$ indicates whether the heater is ``locked'', in which case we cannot switch on or switch off the heater.
Formally, if $\ell_k = 1$ then $s_{k+1} = s_{k}$ necessarily holds.
Hence, $\ell_k$ exposes a physical constraint of the heater, namely that it cannot (or should not) be switched on and off with arbitrarily high frequency.
In the typical case where the minimum switching period of the heater is (much) larger than the \cl's control period ($\approx 100$ ms), the heater will ``lock'' immediately after a switch, i.e., assuming $\ell_k=0$, setting $s_{k+1}$ such that $s_{k+1}\neq s_k$ will induce
$\ell_{k'} = 1$ for every $k' \in [k+1, k + K]$, after which $\ell_{k + K+1} = 0$. Here, $K\in \natnum$ represents the minimum number of timesteps  for which the heater cannot change its state from on to off or \emph{vice versa}.

Suppose that the heater is placed in a room, and that the temperature of this room is a scalar quantity. (We do not aim here to model heat convection through the room or anything like that.)
The temperature in the room, denoted as $T_k$, should remain within predefined ``comfort'' bounds,
\begin{align*}
T_k \in [ \Tmin,\Tmax] \quad k=1,2,\ldots
\end{align*}
where $\Tmin,\Tmax \in \reals$.
If $T_k$ is outside this interval (and only if $\ell_k=0$), then the resource agent should take the trivial action, i.e.,
ensure that the heater is active if $T_k < \Tmin$, and
 inactive if $T_k > \Tmax$.
The more interesting case is 
if $T_k$ lies in $[\Tmin,\Tmax]$ (again, provided that $\ell_k=0$), as this gives rise to some \emph{flexibility} in the heating system:
the degree of freedom here is whether to switch the heater on or off, which obviously directly corresponds to the
total power consumed by the heating system.
The goal is to
delegate this choice to 
the grid agent, 
which we can accomplish by defining an appropriate \cl advertisement.
\paragraph{Defining the Advertisement and the Rounding Behavior} 
\label{sec:radef}
Let the discrete set of implementable real-power setpoints 
at time $k$ be defined as
\[
  \mcal{I}_k\defas  \begin{cases}
  \{ 0 \} & \text{if }(\ell_k=0 \land T_k >\Tmax) \lor \\
          & \phantom{\text{if }}( \ell_k = 1 \land s_k=0),\\
    \{-P_\textnormal{heat},0\} &\text{if }\phantom{(}\ell_k=0 \land \Tmin\leq T_k \leq\Tmax, \\
 \{ -P_\textnormal{heat} \} & \text{if } (\ell_k=0 \land T_k <\Tmin) \lor \\
                           & \phantom{\text{if }}( \ell_k = 1 \land s_k=1),
  \end{cases}
\]
where $\land$ and $\lor$ stand for ``and'' and ``or'', respectively.
Note that $\mcal{I}_k$ only contains non-positive numbers, 
by the convention in \cl that \emph{consuming} real power corresponds to \emph{negative} values for $P$.
We define the \pqprof as in \refthm{discreteres}, with $\mcal{S}_{i_k}$ (as appearing in \refthm{discreteres}) equal to $\mcal{I}_k$, i.e., 
  \[
  \mcal{A}_k\defas  \conv(\mcal{I}_k \times \{0\}) \subset \reals^2, \quad \forall k\in \posint.
\]
We adopt \eqref{eq:discreterealrule} as the rule to compute $u_k^\text{imp}$.
As can be seen from \eqref{eq:discreterealrule}, the relation between the implemented setpoint $u^\text{imp}_k$ and $u^\text{req}_k$ is deterministic.
%
We can expose this relation to the grid agent by means of a \belieffunc (valid for timestep $k$), 
\begin{align*}
 \text{BF}_k:\quad \mcal{A}_k &\rightarrow \powerset{\reals^2} \\
  (p,q) & \mapsto \{ ( \mathrm{proj}_{\mcal{I}_k}( p - e^P_{k-1}), 0) \}.
\end{align*}

We leave the choice of a cost function to the designer of an actual resource agent, because: a) our theorems and the above resource-agent design are independent of this choice, and b) such choice typically depends on scenario-specific details.

\paragraph{Upper Bound on the Accumulated Error}
From \refthm{discreteres} we immediately get the following corollary.

\begin{corollary}
  The single-heater resource agent as defined in \refsec{radef} has 
  $\tfrac12 \Pheat$-bounded accumulated-error.
  \label{cor:sing}
\end{corollary}
\begin{remark}
  The bound given in \refcor{sing} 
  is tight in the sense that
  we can construct hypothetical cases 
  in our single-heater scenario for which $|e_k|=|e^P_k|= \tfrac12\Pheat$ for some $k$.
  For example,
  take $k=1$, suppose that $\Tmin \leq T_1 \leq \Tmax$ so that $\mcal{I}_1 = \{ -\Pheat, 0\}$ and let $\Preq_1 =-\tfrac12 \Pheat $. It then follows that $\Pimp_1= \mathrm{proj}_{\mcal{I}_1} (\Preq_1- e^P_0)=\mathrm{proj}_{\mcal{I}_1} (\Preq_1)  \in \mcal{I}_1$,
  which gives $|e^P_1| =| \Pimp_1 - \Preq_1| = \tfrac12 \Pheat$.
\end{remark}

\subsubsection{General Case: an Arbitrary Number of Heaters}
\label{sec:moreheaters}
Here, we extend the single-heater case to a setting with $r$ heaters, for $r\in \natnum, r\geq 1$ arbitrary. 
As we will see, also this multi-heater case can be analyzed using the tools introduced at the beginning of \refsec{drr}.

Like in the single-heater case, we assume that each heater is purely resistive. We furthermore assume that heater $i$ consumes $P_i^\text{heat}$ Watts of power when active (and zero power when inactive), for every $i\in [r]$.
Also similarly to the single-heater case, we assume that each heater is placed in a separate room, whose (scalar) temperature is denoted as $T_k^{(i)}$.
Not surprisingly, our objective
shall now be to keep the temperature in each room within the predefined comfort bounds,
i.e.,
\begin{align*}
  T_k^{(i)} \in [ \Tmin,\Tmax] \quad \forall k\in \posint, \forall i \in [r].§
\end{align*}

In the one-heater case, the only degree of freedom that can be present is the choice to switch that heater on or off.
In case of multiple heaters, there is potentially some freedom in
choosing which subset of the heaters to activate,
and note that there will typically\footnote{Provided that not too many heaters are locked.} be an exponential number of those subsets (exponential in the number of heaters).
Each subset corresponds to a certain total power consumption, i.e., a power setpoint. 
As in the single-heater case,
the \pqprof will be defined as the convex hull of the collection of these setpoints.
When the grid agent requests some setpoint from the \pqprof, the resource agent has to select an appropriate subset whose corresponding setpoint is closest (in the Euclidean sense) to the requested setpoint.
Note that there can be several subsets of heaters that correspond to the same setpoint.
A simple method to resolve this ambiguity would be, for example, to choose the subset consisting of the coldest rooms, however, as this topic is beyond the scope of this work, we leave the choice of such a selection method to the resource-agent designer.

\paragraph{Characterizing the Set of Implementable Setpoints}
When going from the single-heater setting to a multiple-heaters scenario,
we merely need to re-define $\mcal{I}_k$, which we will name  
$\widetilde{\mcal{I}}_k$ here to avoid confusion with the single-heater case. 
The definition of the \pqprof, \belieffunc, and rule for computing $u_k^\text{imp}$ given in \refsec{radef}
also apply to the multi-heater case, provided that
all occurrences of $\mcal{I}_k$ in those definitions are replaced by $\widetilde{\mcal{I}}_k$.




For every $i\in [r]$, let $s_k^{(i)}$ and $\ell_k^{(i)}$ represent the state variables $s_k$ and $\ell_k$ (as defined in the single-heater case) for the $i$-th heater. Let $\mcal{L}_k \defas  \{ i \in [r]: \ell_k^{(i)} = 1 \}$ denote the set of rooms whose heater is locked at timestep $k$. Furthermore, let $\mcal{C}_k\defas  \{ i \in [r]: T_k^{(i)} < \Tmin \}$
and $\mcal{W}_k\defas  \{ i \in [r]: \Tmin \leq T^{(i)}_k \leq \Tmax \} $. 
Informally speaking, $\mcal{C}_k$ contains the rooms that are ``too cold'', and $\mcal{W}_k$ the rooms whose temperatures are within the comfort bounds.

If $A \subseteq[r]$, we write $\overline{A}$ for the complement
with respect to $[r]$, i.e. $\overline{A}\defas [r]\setminus A$.

Let
\begin{equation} \label{eqn:B_k}
  \widetilde{\mcal{I}}_k : = \big\{ a_k - \sum_{i \in \mcal{S} }  P_i^\textnormal{heat}  : \mcal{S} \subseteq \overline{\mcal{L}_k}\cap\mcal{W}_k \big\}
\end{equation}
represent the set of implementable (active) power setpoints, with
\[
a_k \defas  -\sum_{i \in \mcal{L}_k} s_k^{(i)} P_i^\textnormal{heat} - \sum_{j \in \overline{\mcal{L}_k} \cap \mcal{C}_k } P_j^\textnormal{heat}.
\]

\newcommand{\tS}{{\widetilde{\mathbb{S}}}}
\paragraph{Upper Bound on the Accumulated Error}
As in the one-heater example, we use \refthm{discreteres} to bound the accumulated error of the resource agent.
To this end, let $\tS$ denote the collection of all possible sets
$\widetilde{\mcal{I}}_k$ \eqref{eqn:B_k}. It is easy to see that this is a finite collection. Further, the maximum stepsize of this collection (\refdef{collstep}) is given by $\Delta_\tS = \max_{i \in [r]} P_i^\textnormal{heat}$. 
This gives us the following corollary.

\begin{corollary}
  The multiple-heaters resource agent has $(\tfrac12 \max_{i \in [r]} P_i^\textnormal{heat})$-bounded accumulated-error. 
\end{corollary}

\section{Uncertain Resources}
\label{sec:uncertain}
In this section, we show how to achieve boundedness of the accumulated error for
\emph{uncertain} resources, as given in Definition \ref{def:res_classes}.
In particular, we focus on resources that are affected by Nature, and hence the relation between the advertised $PQ$ profile and the set of implementable setpoints is uncertain. This covers such resources as PV panels, wind farms, and partially controllable loads.

\subsection{Key Property: Projection-Translation Invariance}
\label{sec:kppti}
In order to state our result, we first introduce and important definition 
that we use in the following.

\begin{definition}[Projection-Translation Invariance] \label{def:invar}
Fix $d\in \natnum$. Let $\Dl \subseteq \reals^d$ be a given convex compact set. We say that a convex set $\Il \subseteq \reals^d$ is a \emph{projection-translation invariant} subset of $\Dl$ if $\Il \subseteq \Dl$ and for every $v \in \Dl$ and $u \in \Il$, it holds that
\[
u + v - \mathrm{proj}_{\Il}(v) \in \Dl.
\]
\label{def:tranprop}
\end{definition}

A helpful interpretation of \refdef{invar}  is to view $\tau:=v - \mathrm{proj}_{\Il}(v)$ as a translation vector in the projection direction of $v$ to $\Il$. Then, projection-translation invariance guarantees that the translation of \mcal{I} over $\tau$ remains contained in \mcal{D}.

It is easy to see that in the one dimensional case ($d = 1$), projection-translation invariance is satisfied for intervals.
\begin{proposition} \label{prop:pti1d}
For any $a \leq c \leq d \leq b$, the interval $\Il = [c, d]$ is a projection-translation invariant subset of the interval $\Dl = [a, b]$.
\end{proposition}

\begin{proof}
Let $v \in \Dl$ and $u \in \Il$. If $v \in \Il$, trivially $u + v - \mathrm{proj}_{\Il}(v) = u \in \Dl$. Now, if $v > d$, we have that
\[
u + v - \mathrm{proj}_{\Il}(v) \leq d + b - d = b, \text{ and}
\]
\[
u + v - \mathrm{proj}_{\Il}(v) \geq c + a - d \geq a,
\]
namely $u + v - \mathrm{proj}_{\Il}(v) \in \Dl$. Similarly, if $v < c$, it holds that
\[
u + v - \mathrm{proj}_{\Il}(v) \leq d + b - c \leq b, \text{ and}
\]
\[
u + v - \mathrm{proj}_{\Il}(v) \geq c + a - c = a,
\]
namely $u + v - \mathrm{proj}_{\Il}(v) \in \Dl$.
\end{proof}
For $d = 2$ and if \mcal{I} is a 
convex polygon, then we can construct a collection of sets $\mathbb{D}_\mcal{I}$ such that
\mcal{I} is guaranteed (by construction) to be a projection-translation-invariant subset of every $\mcal{D} \in \mathbb{D}_\mcal{I}$ (see Construction~\ref{constr} in \refsec{ptinvsets} for details).

Constructing sets with the projection-translation invariant property
in higher dimensions in left for further work. It might well be that our construction for polygons carries over to polytopes in $\reals^d$ for arbitrary $d$.


\subsection{Main Results}
The following result provides an algorithm for uncertain resources that are characterized by the projection-translation invariance property. The proof is deferred to \refsec{gen}.

\begin{thm}
Consider an uncertain resource as per \refdef{res_classes}.
In particular,
the set of implementable (feasible) setpoints at time step $k \in \posint$ is given by 
a \term{convex} set $\Il_k$ that is \term{not known at the time of advertising $\Al_k$}.
Suppose in addition that, for every $k$,  $\Il_k$ is a projection-translation invariant subset of a given convex compact set $\Dl \subseteq \reals^2$.
For every $k \in \natnum$, let $e_k \in \reals$ be the accumulated error as defined in \refeq{accerror}.
Then, if the resource agent
\begin{itemize}
\item uses a \term{persistent predictor} to advertise the $PQ$ profile, namely sends $\Al_k = \Il_{k-1}$; and
\item implements $u^{\textrm{imp}}_k = \mathrm{proj}_{\mcal{I}_k}\l(u^{\textrm{req}}_k - e_{k-1}\r)$ for every $k\in \posint$,
\end{itemize}
then it has $(\diam \Dl)$-bounded accumulated-error, where $\diam \Dl$ is the diameter of $\Dl$.
\label{thm:uncertdevices}
\end{thm}

Note that \refthm{uncertdevices} uses the projection-translation-invariance property in a ``many-to-one'' relation: many sets $\mcal{I}_k$ are required to be projection-translation-invariant subsets of one set \mcal{D}. The theorem 
does not say how the set \mcal{D} can be constructed; it simply supposes that such a set \mcal{D} is ``given''.

In \refsec{pvexample}, we will see an example where the sets $\mcal{I}_k$ are selected from a particular parameterized collection of sets for which it turns out to be 
straightforward to find a set $\mcal{D}$ with the required property. 
And, in \refsec{ptinvsets}, we deal with the problem of finding such \mcal{D} for a more general case.

In the one-dimensional case, it is trivial to construct a set $\mcal{D}$ with the required ``many-to-one'' property: using Proposition \ref{prop:pti1d}, it is easy to see that such a $\mcal{D}$ can be constructed by taking the union of all intervals $\mcal{I}_k$; we will use this in the following corollary.

\begin{corollary}
  Consider an uncertain \temph{real-power only} resource. In particular, the set of implementable (feasible) setpoints at time step $k \in \posint$ is given by an \term{interval}  $\Il_k$ that is \term{not known at the time of advertising $\Al_k$}.
Then, if the resource agent
\begin{itemize}
\item uses a \term{persistent predictor} to advertise the $PQ$ profile, namely sends $\Al^P_k = \Il_{k-1}$ ($\Al^Q_k = \{0\}$); and
\item implements $P^{\textrm{imp}}_k = \mathrm{proj}_{\mcal{I}_k}\l(P^{\textrm{req}}_k - e^P_{k-1}\r)$ for every $k\in \posint$,
\end{itemize}
then it has $(\diam \Dl)$-bounded accumulated-error, where $\Dl :=\conv \l(\bigcup_{k = 1}^{\infty} \Il_k \r)$.
\end{corollary}

\subsection{Example: Photovoltaic (PV) System}
\label{sec:pvexample}
Here, we explain how we can apply \refthm{uncertdevices} to
devise a resource agent for a PV system with bounded accumulated-error.

Let $S_\text{rated}$ and $\phi_{\max}$ denote
the rated power of the converter and angle corresponding to the minimum power factor, respectively. We suppose that these quantities are given (they correspond to physical properties of the PV system), and that $S_\text{rated}\geq 0$ and $0 \leq \phi_{\max} < \pi$. 
Note that for any power setpoint $(P,Q)$, the rated power imposes the constraint $P^2+Q^2 \leq  S^2_\text{rated}$; the angle $\phi_{\max}$ imposes that $\arctan (Q/P) \leq \phi_{\max}$.

Let us now choose $P_{\max},  \varphi \in \reals$ such that $0 \leq P_{\max} \leq S_\text{rated}$, $ 0 \leq \varphi \leq \phi_{\max}$, and $\frac{P_{\max}}{\cos \varphi} = S_\text{rated}$, and let
\[
  \mcal{T}(x) :=  \{ (P,Q)\in \reals^2 : 0 \leq P \leq x, \frac{|Q|}{P}\leq \tan \varphi \}
\]
be a triangle-shaped set in the $PQ$ plane; see \reffig{Tset} for an illustration.
Note that for any combination of $P_{\max}$ and $\varphi$, the triangle $\mcal{T}(x)$ for any $x\in [0,P_{\max}]$ is fully contained in the disk that corresponds to the rated-power constraint, and, moreover, the two upper corner points of $\mcal{T}(P_{\max})$  lie on the boundary of that disk.

Let $p^{\max}_k$ be the maximum real power available at timestep $k\in \natnum$ (typically determined by the solar irradiance). 
Using $p^{\max}_k$, we define the set of implementable points
at timestep $k$ as
\[
\mcal{I}_k:= \mcal{T}(\min(p_k^{\max},P_{\max})).
\]
\begin{figure}
  \centering
  \includegraphics[scale=.8]{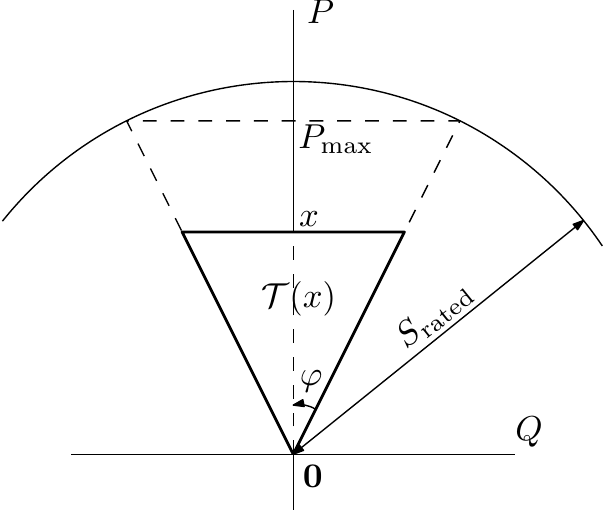}
  \caption{The parameterized collection of sets $\{ \mcal{T} (x): x\in [0,P_{\max}]\}$, and its relation to the rated-power constraint of the PV converter.}
  \label{fig:Tset}
\end{figure}
\begin{corollary}
Let $P_{\max}$, $\varphi$, and $\mcal{I}_k$ (for every $k\in \natnum$) be defined as above for a given PV system.
    A resource agent for this PV system
    that advertises $\mcal{A}_k := \mcal{I}_{k-1}$
  and implements $u^{\textrm{imp}}_k = \mathrm{proj}_{\mcal{I}_k} \l(u^{\textrm{req}}_k - e_{k-1}\r)$ for every $k\in \posint$
  has $c$-bounded accumulated-error with
  \[
c:=\max\l\{ \frac{P_{\max}}{\cos\varphi}, 2\, P_{\max}\tan \varphi \r\}.
  \]
  \label{cor:pvcor}
\end{corollary}
The main ingredient of the proof of \refcor{pvcor} is the following lemma,
whose proof is given in \refsec{proofoftriang}.
\begin{lemma}
Every set in the collection $\{ \mcal{T} (x): x\in [0,P_{\max}]\}$ is projection-translation-invariant with respect to $\mcal{D}:= \mcal{T} (P_{\max})$.
\label{lem:triangleprop}
\end{lemma}
\begin{proof}[Proof of Corollary~\ref{cor:pvcor}.]
  By definition of $\mcal{I}_k$ and by \reflem{triangleprop}, each member of the sequence $(\mcal{I}_k)_k$ is projection-translation-invariant with respect to $\mcal{T}(P_{\max})$.
  Because the resource agent computes the \pqprof $\mcal{A}_k$ and $u^\text{imp}_k$ in accordance with the requirements of \refthm{uncertdevices}, we can apply the latter to conclude that the resource agent has $\mathrm{diam}\mcal{T}(P_{\max})$-bounded accumulated-error. Since $\mcal{T}(P_{\max})$ is an isosceles triangle, its diameter is either $P_{\max}/\cos \varphi$ (the length of one of its \emph{legs}) for $\varphi \leq\frac{\pi}{6}$ or $ 2\,P_{\max}\tan \varphi$ otherwise (the length of its \emph{base}).
Hence, the claim follows.
\end{proof}

\section{Achieving Bounded Accumulated-Error through Error Diffusion: Generic Approach and Proofs} \label{sec:gen}
In this section, we propose a generic approach for achieving bounded accumulated-error via the error-diffusion method. The results in this section are used to prove the results of Sections
\ref{sec:drr} and \ref{sec:uncertain}; however, they can also be used to design agents for other resources that satisfy appropriate conditions.
The reader that is interested in the performance of the RAs proposed in Sections
\ref{sec:drr} and \ref{sec:uncertain} can safely skip the present section and proceed to \refsec{num}.

We next introduce two general concepts: the first (\refdef{bounded_g_error}) is a generalization of the $c$-bounded accumulated-error concept, and the second (\refdef{Gapprox}) is a generalization of the projection concept.
\begin{definition}
  \label{def:bounded_g_error}
  For any set $\mcal{G}\in \mcal{B}(\reals^2)$,
  we say that a given resource agent has \term{\mcal{G}-contained accumulated-error} if 
\[
  e_{ k} \in \mcal{G}, \quad k \geq 1.
\]
\end{definition}
Note that having \mcal{G}-contained accumulated-error 
implies having $c$-bounded accumulated-error (\refdef{bounded_error}) with $c \geq \max \{ \| x \| : x\in \mcal{G}\}$.



\begin{definition} \label{def:Gapprox}
Fix $d\in \natnum$. Let $F: \mcal{D} \rightarrow \mcal{I}$ be an arbitrary map, where $\mcal{D},\mcal{I} \subseteq \reals^d$.
We say that $F$ is a \term{\map{\mcal{G}}} if there exists
a set $\mcal{G} \in \mcal{B}(\reals^d)$ 
such that
$$ F(x) - x \in \mcal{G}$$
holds for all $x \in \mcal{D}$.
\end{definition}

Note that the projection operator is a \map{\mcal{G}} for suitably chosen $\mcal{G}$.

The following lemma is a key result that explains why the error-diffusion method leads to bounded accumulated-error.

\begin{lemma}
Let $\mcal{G} \in \mcal{B}(\reals^2)$ and let $k \in \posint$ be arbitrary.
Let $\mcal{D}_k, \mcal{I}_k \subseteq \reals^2$ be non-empty sets and let
$F_k:\mcal{D}_k \rightarrow \mcal{I}_k$ be a \map{\mcal{G}}.
Let $e_k \in \reals^2$ be the accumulated error as defined in \refeq{accerror}. 
Let $u^{\textrm{req}}_k \in \reals^2$. 
Then, if
\begin{equation} \label{eq:asm_req}
u^{\textrm{req}}_k - e_{k-1} \in \Dl_k
\end{equation}
holds, 
and
\begin{equation}
u^{\textrm{imp}}_k  \defas F_k\l(u^{\textrm{req}}_k - e_{k-1}\r), 
\label{eq:asm_imp}
\end{equation}
we have that
\[
  e_k \in \mcal{G}. 
\]
\label{lem:berror}
\end{lemma}
\begin{proof}
Note that we can write $e_k$ in the following equivalent recursive form:
\begin{align}
  e_{k} & = e_{k-1} + u^\text{imp}_k - u^\text{req}_k, \notag \\ 
          & = e_{k-1} +F_k( u^\textnormal{req}_k - e_{k-1})  - u^\textnormal{req}_k \notag \quad\text{(using \refeq{asm_imp})} \\
 \label{eqn:e_rec}       & =F_k( u^\textnormal{req}_k - e_{k-1})  - (u^\textnormal{req}_k - e_{k-1}).
\end{align}
The result follows since $F_k$ is a \map{\mcal{G}}.
\end{proof}

Now, the main idea of our approach is as follows: if (i) a resource agent can be characterized by a \emph{sequence} of \map{\mcal{G}}{s} $F_k, k\in\posint$ for some $\mcal{G}\in \mcal{B}(\reals^2)$ that is independent of $k$, where this sequence of maps $\{F_k\}$ represents the physics of the controlled resource as well as any internal (low-level) control loops in the resource itself or in the resource agent,
and (ii) $u^\textnormal{imp}_k$ is given by 
\refeq{asm_imp} 
for every $k \in \posint$,
then that resource agent achieves \mcal{G}-contained accumulated-error
if we can invoke \reflem{berror} for every $k \in \posint$.
Conditioned on (i) and (ii), we may invoke \reflem{berror} for every $k$ if we can prove that condition \eqref{eq:asm_req} is satisfied for all $k$.
In the following subsections, we show how we can ensure \eqref{eq:asm_req} for deterministic and uncertain resources.


\subsection{Result for Deterministic Resources}

In this section, we focus on deterministic resources as per Definition \ref{def:res_classes}. In practice this means that the set of implementable setpoints at time step $k$ is known in advance, hence can be used to advertise the $PQ$ profile $\Al_k$.

\begin{thm} \label{thm:ball}
Let $\mcal{G} \in \mcal{B}(\reals^2)$.
Let $e_k \in \reals^2$ be the accumulated error as defined in \refeq{accerror} for every $k\in\natnum$.
For every $k \in \posint$,  let $\mcal{I}_k \in \mcal{B}(\reals^2)$, $\mcal{A}_k:=\conv (\mcal{I}_k)$ and  $u^{\textrm{req}}_k \in \mcal{A}_k$,
let $\Dl_k := \Al_k + \mcal{G}$, 
and let $F_k: \mcal{D}_k \rightarrow \mcal{I}_k$ a \map{\mcal{G}}.
Then, if a resource agent implements $u^{\textrm{imp}}_k = F_k\l(u^{\textrm{req}}_k - e_{k-1}\r)$ for every $k\in \posint$,  
it has \mcal{G}-contained accumulated-error.
\end{thm}

The trick here is to define $\Dl_k$ (the domain of $F_k$) as the ``\mcal{G}-inflation'' of $\Al_k$, such that condition \eqref{eq:asm_req} is trivially satisfied.
Of course, to use this theorem for an actual application, it remains to ensure that the (application-specific) map $F_k$ is a \map{\mcal{G}} on this particular domain.
In \refsec{proofDiscr}, we prove that for the discrete resources of \refsec{drr}, the corresponding map has this required property.


\begin{proof}
We will prove the statement using induction on $k$. 
By definition of $e_k$, the statement holds for $k=0$.
Suppose that the statement holds for timestep $k-1$ (for arbitrary $k \in \posint$).
Since $u^\textnormal{req}_k \in \Al_k$ and because
$e_{k-1} \in \mcal{G}$ (by the induction hypothesis), it holds that
$$u^\textnormal{req}_k - e_{k-1} \in \Al_k+ \mcal{G} = \Dl_k $$ by construction. Furthermore, $F_k$ is a \map{\mcal{G}} and
$u^{\textrm{imp}}_k = F_k\l(u^{\textrm{req}}_k - e_{k-1}\r)$, hence we may invoke \reflem{berror} to conclude that $e_k \in \mcal{G}$. 
\end{proof}

\begin{remark}
From the proof, it might seem that  we do not require any relation between $\mcal{A}_k$ and $\mcal{I}_k$, however, without any such relation it will be impossible to construct $F_k$ such that it is a \map{\mcal{G}} for every $k$.  We define $\mcal{A}_k:=\conv (\mcal{I}_k)$ as it satisfies our needs, but in principle one could be more general here, and define $\mcal{A}_k$ as an arbitrary function of $\mcal{I}_k$.
\end{remark}

\subsection{Proof of \refthm{discreteres}} \label{sec:proofDiscr}

\begin{fact} 
  The bound
\[
|\mathrm{proj}_{\mcal{S}}(x) - x| \leq \frac{\Delta_\mcal{S}}{2}
\]
holds for every finite non-empty set $\mcal{S}\subset \reals$ and for every $x \in [\min \mcal{S},\max\mcal{S}]$.
\label{fact:quanterr}
\end{fact}

\begin{lemma}
  Let $\mathbb{S}$ be a finite non-empty collection of sets $\mathbb{S}\defas \{ \mcal{S}_i \}_{i \in [n]}$ where $n\in \posint$ and $\mcal{S}_i  \subset \reals$ is a finite non-empty set for all $i\in [n]$. Let $\Delta_{\mathbb{S}}$ denote the maximum stepsize of $\mathbb{S}$, and define $\mcal{G}:= [-\tfrac12\Delta_{\mathbb{S}},\tfrac12\Delta_{\mathbb{S}}]\times \{0\}$.
  Let $(i_k)_{k \in \posint}  $ be any sequence with $i_k \in [n]$, and let $\Il_k \defas \mcal{S}_{i_k}\times \{0\}$,
$\mcal{D}_k := \conv(\mcal{I}_k ) + \mcal{G}$ and
\begin{align*}
  F_k:  \mcal{D}_k &\rightarrow \mcal{I}_k \\
  (p,q) & \mapsto (\mathrm{proj}_{\mcal{S}_{i_k}}(p),0)
\end{align*}
be defined for every $k\in \posint$.
Then $F_k$ is a \map{\mcal{G}} 
for all $k \in \posint$.
\label{lem:apprmap}
\end{lemma}
\begin{proof}
We need to show that $F_k(x) - x \in \mcal{G}$ for every $x \in \mcal{D}_k$ and every $k\in \posint$.
Since $\mcal{D}_k^{Q} = \mcal{G}^{Q}=0$, it suffices to show that $\mathrm{proj}_{\mcal{I}_k}(y)-y \in \mcal{G}^{P}=[-\tfrac12\Delta_{\mathbb{S}},\tfrac12\Delta_{\mathbb{S}}]$ for every $y \in \mcal{D}_k^{P}= [ \min \mcal{I}_k - \tfrac12\Delta_{\mathbb{S}} \ , \max \mcal{I}_k +  \tfrac12\Delta_{\mathbb{S}}]$ and every $k\in \posint$.

  Let $k$ be arbitrary.
If $y \in [\min \mcal{I}_k, \max \mcal{I}_k]$, then
\[
|\mathrm{proj}_{\mcal{I}_k}(y) -y |
\leq \frac{\Delta_{\mcal{S}_{i_k}}}{2}\leq \frac{\Delta_\mathbb{S}}{2}
\]
by \reffact{quanterr} and by \refdef{collstep}.
Otherwise, i.e., if $y\notin \conv(\Il_k)$, $y$ will be projected to the closest boundary of $\conv(\Il_k)$, 
for which $|\mathrm{proj}_{\mcal{I}_k}(y) -y | \leq \tfrac12\Delta_\mathbb{S}$ holds by definition of $\mcal{D}_k$.
\end{proof}

\begin{proof}[Proof of Theorem \ref{thm:discreteres}]
  \reflem{apprmap} guarantees that the map
  \[
    F_k( u^\text{req}_k):=(\mathrm{proj}_{\mcal{S}_{i_k}}(P^\text{req}_k - e^P_{k-1}),0)
  \]
is a \map{\mcal{G}} on the domain $\mcal{A}_k+\mcal{G} $ where $\mcal{A}_k:=\conv(\mcal{I}_k)$,
  $\mcal{I}_k:=\conv ( \mcal{S}_{i_k}\times \{0\})$ and
  $\mcal{G}:=[-\tfrac12\Delta_{\mathbb{S}},\tfrac12\Delta_{\mathbb{S}}]\times \{0\}$.
Because the implemented setpoint is computed as
$$u^{\textrm{imp}}_k =
(\mathrm{proj}_{\mcal{S}_{i_k}}(P^\text{req}_k - e^P_{k-1}),0)$$ for every $k\in \posint$,
we can apply \refthm{ball} to conclude that the resource agent has \mcal{G}-contained accumulated-error, which implies that the resource agent has $\tfrac12\Delta_{\mathbb{S}}$-bounded accumulated-error by definition of \mcal{G}.
\end{proof}

\subsection{Proof of \refthm{uncertdevices}}

\begin{proof}
Consider the mapping $F_k : \mcal{D} \rightarrow \mcal{I}_k$, $x \mapsto \mathrm{proj}_{\mcal{I}_k}(x)$. Recall that $\Il_k$ is, in particular, a convex subset of $\Dl$ by the definition of projection-translation invariance. 
Thus, $\| F_k(x) - x \| \leq \diam \Dl$ holds for all $x \in \Dl$ and all $k\in \posint$; namely, $F_k$ is a \map{$(\diam \Dl)$} for all $k\in \posint$.

Next, we prove that (i)
$u^{\textrm{req}}_k - e_{k-1} \in \mcal{D}$ and that (ii) $\|e_k\|\leq \diam \Dl$  holds for all $k \in \posint$,  using induction.
Clearly, claim (i) holds for $k = 1$ as $u^{\textrm{req}}_1 - e_0 = u^{\textrm{req}}_1 \in \mcal{A}_1 = \Il_0 \subseteq \mcal{D}$. This result allows us to invoke \reflem{berror} and conclude that (ii) also holds for $k=1$.

Now suppose that $u^{\textrm{req}}_k - e_{k-1} \in \Dl$ holds
for arbitrary $k \in \posint$ (the induction hypothesis).
For $k+1$, we have, as in \eqref{eqn:e_rec},
\[
u^{\textrm{req}}_{k+1} - e_{k}  = u^{\textrm{req}}_{k+1} + (u^\textnormal{req}_k - e_{k-1}) - \mathrm{proj}_{\mcal{I}_k}( u^\textnormal{req}_k - e_{k-1}).
\]
Observe that $u^{\textrm{req}}_{k+1} \in \Il_k$ as the resource agent uses the persistent predictor for the $PQ$ profile. Also, by the induction hypothesis, $u^\textnormal{req}_k - e_{k-1} \in \Dl$. Hence, invoking the projection-translation invariance property of $\Il_k$ (Definition \ref{def:invar}) with $u \equiv u^{\textrm{req}}_{k+1} $ and $v \equiv u^\textnormal{req}_k - e_{k-1}$, we obtain that $u^{\textrm{req}}_{k+1} - e_{k} \in \Dl$, which proves (i) for $k+1$. Again, this allows us to invoke \reflem{berror} for $k+1$ and conclude that $\|e_{k+1}\|\leq \diam \Dl$ which completes the induction argument, and with that, the proof.
\end{proof}

\subsection{Projection-Translation Invariance: Constructing Supersets of Polygons}
\label{sec:ptinvsets}
In this section we present, for a given convex polygon \mcal{I}, how to construct a superset \mcal{D} such that $\mcal{I}$ is a projection-translation invariant subset of \mcal{D}.
The superset is not unique; in fact, the degrees of freedom in constructing the superset give rise to an (infinite) collection of supersets, $\mathbb{D}_{\mcal{I}}$, which we have already introduced in \refsec{kppti}.

We will use the construction in the following way: when given a \emph{collection} of sets $\{\mcal{I}_k\}_k$, and if, for every $k$, $\mcal{I}_k$ is a convex polygon, then we can construct a \emph{minimal} set \mcal{D} with the ``many-to-one'' property (which means that for every $k$, $\mcal{I}_k$ is a projection-translation invariant subset of \mcal{D}) as 
\begin{equation}
    \mcal{D} \in \arg\min \{\diam \mcal{W} : \mcal{W} \in \bigcap_{k} \mathbb{D}_{\mcal{I}_k}\},
    \label{eq:inters}
  \end{equation}
  where we claim that the intersection in \eqref{eq:inters} is always non-empty and always contains a set with bounded diameter, provided that all sets $\mcal{I}_k$ are bounded.

 For our construction, we need the following definitions. (See also \reffig{cones} for some pictorial examples of these definitions.) 
\begin{definition}[Cone]
We define the (convex) \emph{cone} of any two vectors $v,w\in \reals^2$ as 
\[
  \mathrm{cone}(v,w):=\{ \alpha v + \beta w: \alpha,\beta \in \reals, \alpha \geq 0, \beta \geq 0 \}.
\]
\end{definition}
\begin{definition}[Polar Cone]
  For any cone $\mcal{C}\subset \reals^2$, we define the \emph{polar cone} of $\mcal{C}$ as
\[
  \mcal{C}^o := \{ y \in \reals^2: y^{\textnormal{T}} x\leq 0 \quad \forall x \in \mcal{C} \}
\]
\end{definition}
\begin{figure}
  \centering
  \includegraphics{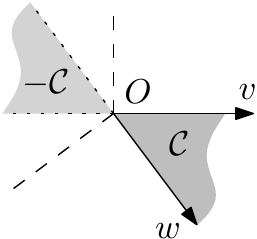}\hspace{2em}
  \includegraphics{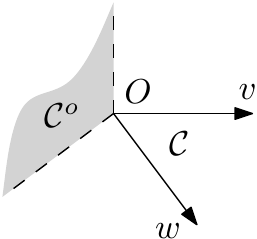}\hspace{2em}
  \includegraphics{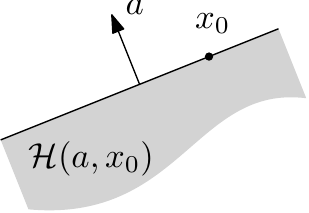}
  \caption{The cones $\mcal{C}:=\mathrm{cone}(v,w)$, $-\mcal{C}$ and $\mcal{C}^o$, and the half-space $\mcal{H}(a,x_0)$.}
  \label{fig:cones}
\end{figure}
\begin{figure}
  \centering
  \includegraphics[scale=.8]{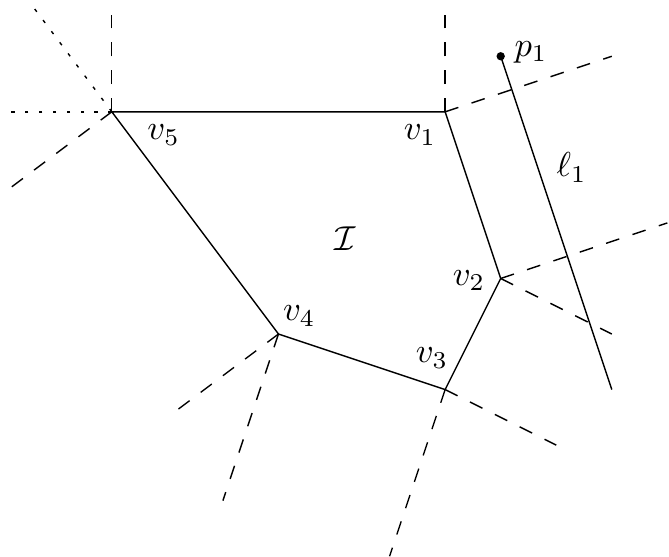}
  \includegraphics[scale=.8]{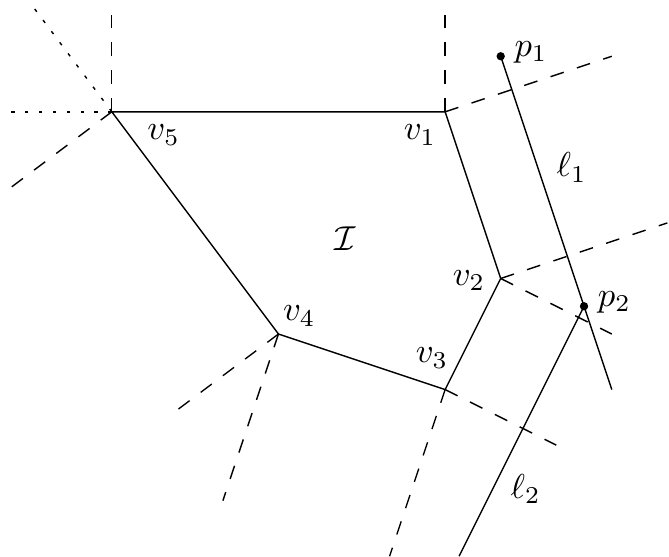}
  \includegraphics[scale=.8]{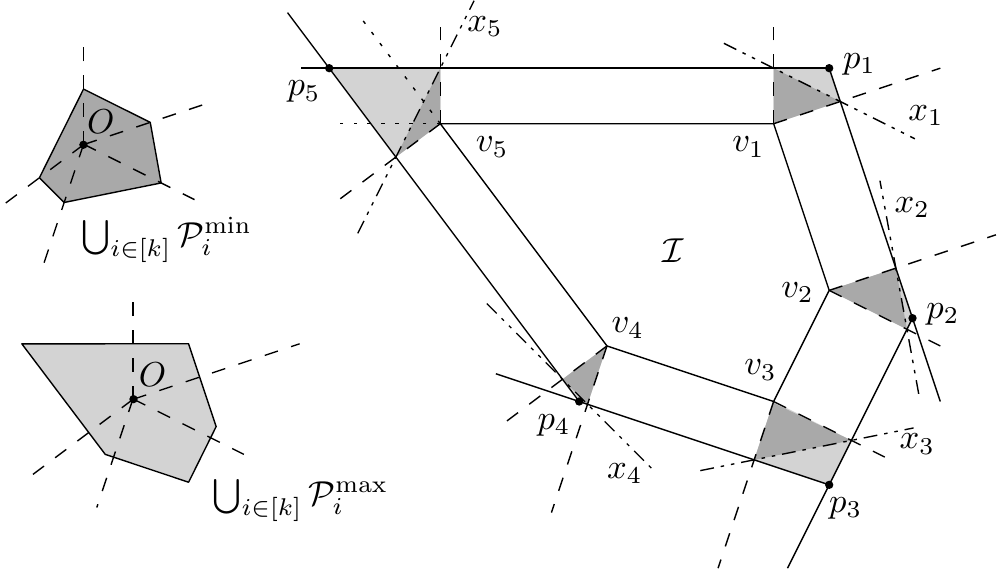}
  \caption{Illustration of parts of Construction~\ref{constr} for constructing a set \mcal{D} such that \mcal{I} is a projection-translation-invariant subset of \mcal{D}.}
  \label{fig:constr}
\end{figure}

\begin{definition}[Half Space]
Let
\begin{align*}
\mcal{H}:\reals^2 \times\reals^2 & \rightarrow \mcal{P}(\reals^2) \\
 (a,x_0)&\mapsto\{ x \in \reals^2 \,|\, a^T(x-x_0)\leq 0 \}.
\end{align*}
\end{definition}
Note that $\mcal{H}(a, x_0)$ is a closed halfspace that extends in direction $-a$ and where $x_0$ lies on its associated halfplane.

\newcommand{\wrap}{\mathrm{wrap}}
Our construction is given below. See \reffig{constr} for a graphical example.
\begin{construction}
  \label{constr}
Let $\mcal{I}\subset \reals^2$ be a simple convex polygon. Let $n$ denote the number of vertices of \mcal{I}, let $v_1$ be an arbitrary vertex of \mcal{I} and let $v_2, \ldots,v_n$ respectively denote the remaining vertices encountered when traversing the polygon clockwise from $v_1$.
Let
\[
\mathrm{wrap}(i):=\begin{cases} 1 & \text{if } i = n+1 \\
i & \text{if } i \in [n] \\
n &\text{if } i =0
\end{cases}
\]
be defined for any $i \in \{0,\ldots,n+1\}$.
 Let $\mcal{C}_i:= \mathrm{cone}(v_{\wrap(i+1)}-v_i,v_{\wrap(i-1)}-v_i)$ for every $i\in [n]$.

Choose an arbitrary point $\pi_1 \in (-\mcal{C}_1) \cap \mcal{C}_1^o$ and define $p_1:=v_1+\pi_1$.\footnote{If the interior angle of vertex $i$ is \emph{acute}, then $(-\mcal{C}_i) \cap \mcal{C}_i^o = -\mcal{C}_i$. Otherwise, i.e., if the interior angle is \emph{obtuse}, then $(-\mcal{C}_i) \cap \mcal{C}_i^o = \mcal{C}_i^o$.} Let $\ell_1$ be the line through $p_1$ and parallel to $v_2-v_1$. For every $j\in \{2,\ldots,n-1\}$, choose a point $\pi_j \in (\ell_{j-1}- \{ v_j\} )\cap (-\mcal{C}_j) \cap \mcal{C}_j^o $, let $p_j:=v_j+\pi_j$ and let $\ell_j$ be the line through $p_j$ and parallel to $v_{j+1}-v_j$.
Finally, let $\ell_n$ be the line through $p_1$ and parallel to $v_n-v_1$, and define $p_n$ as the intersection between $\ell_{n-1}$ and $\ell_n$.

Let $\mcal{P}_i^{\min}:= \mcal{C}_i^o \cap \mcal{H}(v_i-\mathrm{proj}_{x_i}(v_i), \mathrm{proj}_{x_i}(v_i)-v_i)$, where $x_i$ is the line through 
$\mathrm{proj}_{\ell_i}(v_i)$ and $\mathrm{proj}_{\ell_{\wrap(i - 1)}}(v_i)$.
Let \mcal{R} be the polygon defined by $\{p_i\}_{i\in n}$, and let $\mcal{P}_i^{\max}:= \mcal{C}_i^o \cap (\mcal{R}-\{ v_i \})$ for every $i\in[n]$.

Let \mcal{G} be an arbitrary convex set such that $\bigcup_{i\in [n]} \mcal{P}^{\min}_i \subseteq \mcal{G} \subseteq \bigcup_{i\in [n]} \mcal{P}^{\max}_i$. Define $\mcal{D}:= \mcal{I}+\mcal{G}$.
\end{construction}
\begin{lemma}
  For a given simple convex polygon \mcal{I}, any set \mcal{D} constructed using Construction~\ref{constr} is such that \mcal{I} is a projection-translation-invariant subset of \mcal{D}.
 \label{lem:polygons}
\end{lemma}
\begin{proof}
  Let $v$ be as defined in \refdef{tranprop}.
  First, note that the projection-translation-invariance property trivially holds when $v \in \mcal{I}$.
  Hence, suppose w.l.o.g. that $v \in \mcal{D}\setminus\mcal{I}$.
  Now, observe that, by construction, the set \mcal{G} contains all possible translation vectors $\{v - \mathrm{proj}_\mcal{I}(v)\}_{ v \in \mcal{D}\setminus\mcal{I}}$. Namely, (i) for every $i\in [k]$ the line $\ell_i$ is parallel to the corresponding facet of \mcal{I}; (ii) $\mcal{G}\supseteq \bigcup_{i\in [k]} \mcal{P}^{\min}_i$ and, (iii) $\mathrm{proj}$ is the orthogonal projection. Hence, \mcal{I} is a projection-translation-invariant subset of $\mcal{D}=\mcal{I}+\mcal{G}$. 
\end{proof}

\subsection{Proof of \reflem{triangleprop}}
\label{sec:proofoftriang}
Throughout the proof (of \reflem{triangleprop}), we re-use the notation and objects that we introduced/defined in Construction~\ref{constr}. \reffig{triangleconstr} shows a helpful illustration of the objects that play a role in the proof.

\begin{proof}[Proof of Lemma~\ref{lem:triangleprop}.]
Pick $x\in [0,P_{\max}]$ arbitrarily. We now use Construction~\ref{constr} to construct \mcal{D} from $\mcal{T}(x)$ (where the latter corresponds to the set \mcal{I} in Construction~\ref{constr}). 
Let $v_1 := (0,0)$ and $n=3$. Choose $p_1 := v_1$. Note that this choice also defines the line $\ell_1$. Choose $p_2$ on $\ell_1$, such that $p_2^P = P_{\max}$. This choice also defines $\ell_2$, $p_3$ and $\ell_3$. Let $\mcal{G}:= \bigcup_{i \in [3]} \mcal{P}^{\max}_i =  \mcal{P}^{\max}_2 \cup \mcal{P}^{\max}_3$, since $\mcal{P}^{\max}_1=\{(0,0)\}\subseteq  \mcal{P}^{\max}_2\cup \mcal{P}^{\max}_3$.
Now, because $\mcal{T}(x)$ is a convex polygon and because we followed Construction~\ref{constr} to construct $\mcal{D}=\mcal{T}(x) + \mcal{G}$, we can use \reflem{polygons} to conclude that $\mcal{T}(x)$ is a projection-translation-invariant subset of \mcal{D}. Because this holds true for arbitrary $x\in [0,P_{\max}]$, we may conclude that every set in the collection $\{ \mcal{T} (x): x\in [0,P_{\max}]\}$ is projection-translation-invariant with respect to $\mcal{D}$.
\end{proof}
\begin{figure}
  \centering
  \includegraphics[scale=.8]{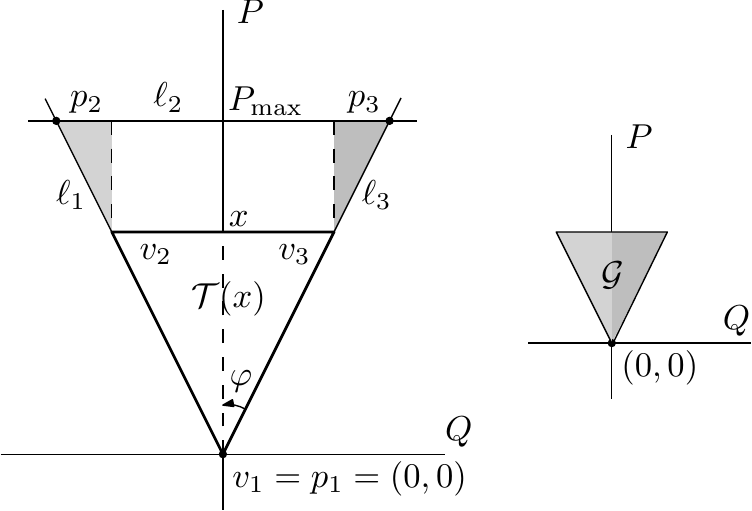}
  \caption{Invoking Construction~\ref{constr} to show that $\{ \mcal{T} (x): x\in [0,P_{\max}]\}$ is a projection-translation-invariant collection with respect to $\mcal{D}:= \mcal{T} (P_{\max})$.}
  \label{fig:triangleconstr}
\end{figure}

\section{Numerical Illustration} \label{sec:num}
Our aim here is to explore through numerical simulation the behavior of the closed-loop system (as introduced in \refsec{ga_opt}) in a realistic scenario that is not covered by Hypothesis 1. Our simulation results indicate that the use of resource agents that have the bounded accumulated-error property
significantly improves the behavior the closed-loop \cl system.

As in \cite{commelec2}, we take a case study that makes reference to the low voltage microgrid benchmark defined by the CIGR\'{E} Task Force C6.04.02 \cite{LVBenchmark}.
For the full description of the case study and the corresponding agents design, the reader is referred to \cite{commelec2}.








There are two modifications compared to the original case study: (i) The PV agents are updated with the algorithm described in Section \ref{sec:uncertain}, and (ii) the uncontrollable load (UL2) is replaced by a controllable load (modeling a resistive heater), and the corresponding agent is implemented according to the methods described in Section \ref{sec:building}.


We simulate a rather extreme scenario involving a highly variable solar irradiance profile. I.e., we let the irradiance vary according to a square wave with a period of  $300$ ms. This will cause the PV agent's \pqprof to be highly variable, which, in turn, means that part (ii) of Hypothesis 1 will get violated frequently.
Furthermore, our setup includes various resource agents that send time-variable cost functions, which violates condition (iii) of Hypothesis 1.
Note that the varying state of the grid will, through the related cost term, also violate condition (iii).

We let the cost function of the PV agent be the same as in \cite{commelec2}; this cost function  encourages to maximize active-power output. The cost function of the heater is set to a quadratic function, whose minimum lies at half the heater power, namely at $-7.5$ kW. With respect to the locking behavior of the heater, we let it lock for one second after a switch.

The results are shown in Figures \ref{fig:errorDiff_heater}--\ref{fig:errorDiffPV}. For comparison, we run the same scenario with resource agents
for which the accumulated error might grow unboundedly. I.e., those RAs do not apply the error-diffusion technique described in this paper, instead, they just project the request to the closest implementable setpoint, like in \cite{commelec2}.
From the results, we see the following benefits:
\begin{itemize}
  \item Improved utilization of renewables, less curtailment. (\reffig{utilization})
  \item Convergence to the optimum of the cost functions. (\reffig{convergence})
  \item Delivery of the requested power on average, i.e., energy, which can be valuable for an application like a virtual power plant. (\reffig{accerror_heater}, \ref{fig:convergence}, \ref{fig:accerror_PV} and \ref{fig:errorDiffPV})
  \item Less sensitivity to the choice of the grid agent's gradient-descent step size. (\reffig{errorDiff_heater})
\end{itemize}

\begin{figure}
\centering
\includegraphics[width=.8\columnwidth]{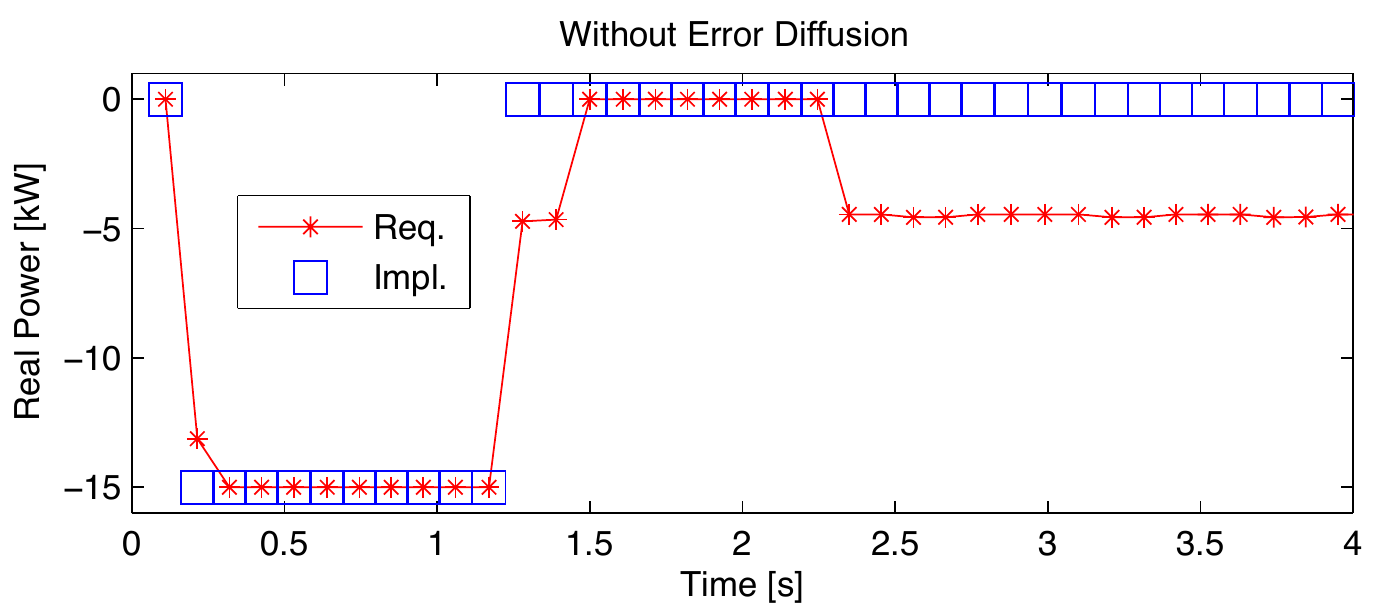}\\
\includegraphics[width=.8\columnwidth]{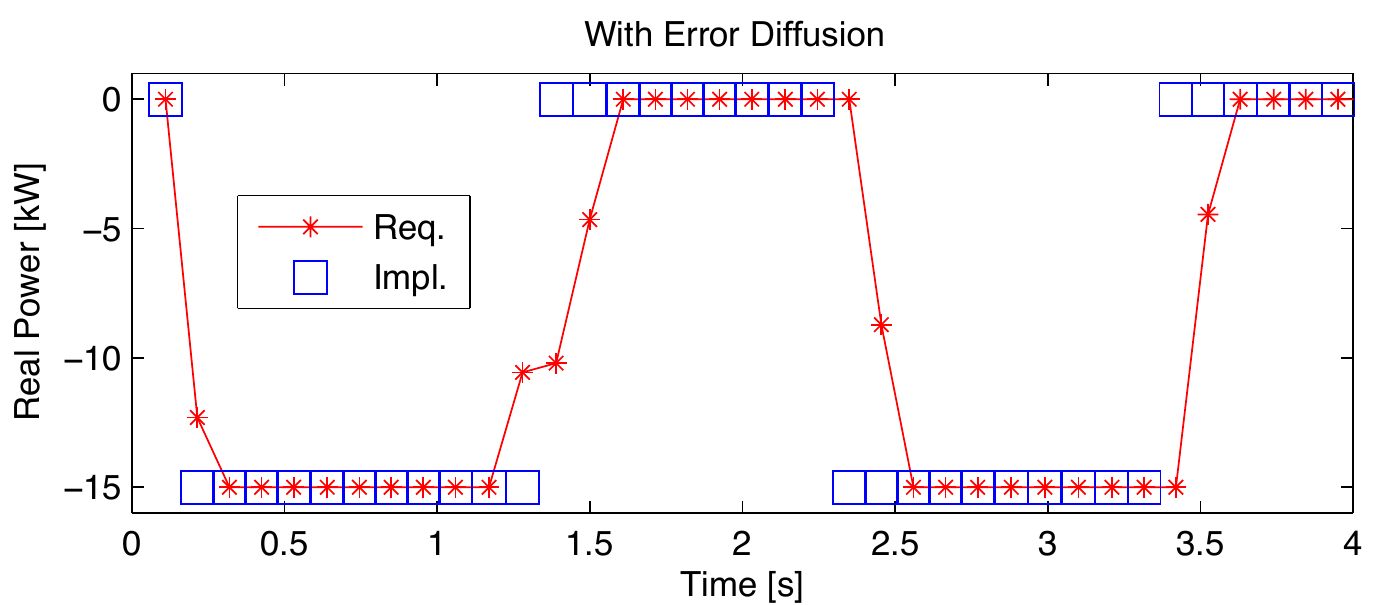}
\caption{Sequence of requested vs. implemented setpoints belonging to the heater agent. In the top figure (no error diffusion) the grid agent ``gets stuck'' on a setpoint that is not close to the optimal value. In the bottom figure (with error diffusion), the heater switches on and off with an appropriate duty cycle (the switching frequency is limited by the locking-duration parameter). }
\label{fig:errorDiff_heater}
\end{figure}

\begin{figure}
\centering
\includegraphics[width=.8\columnwidth]{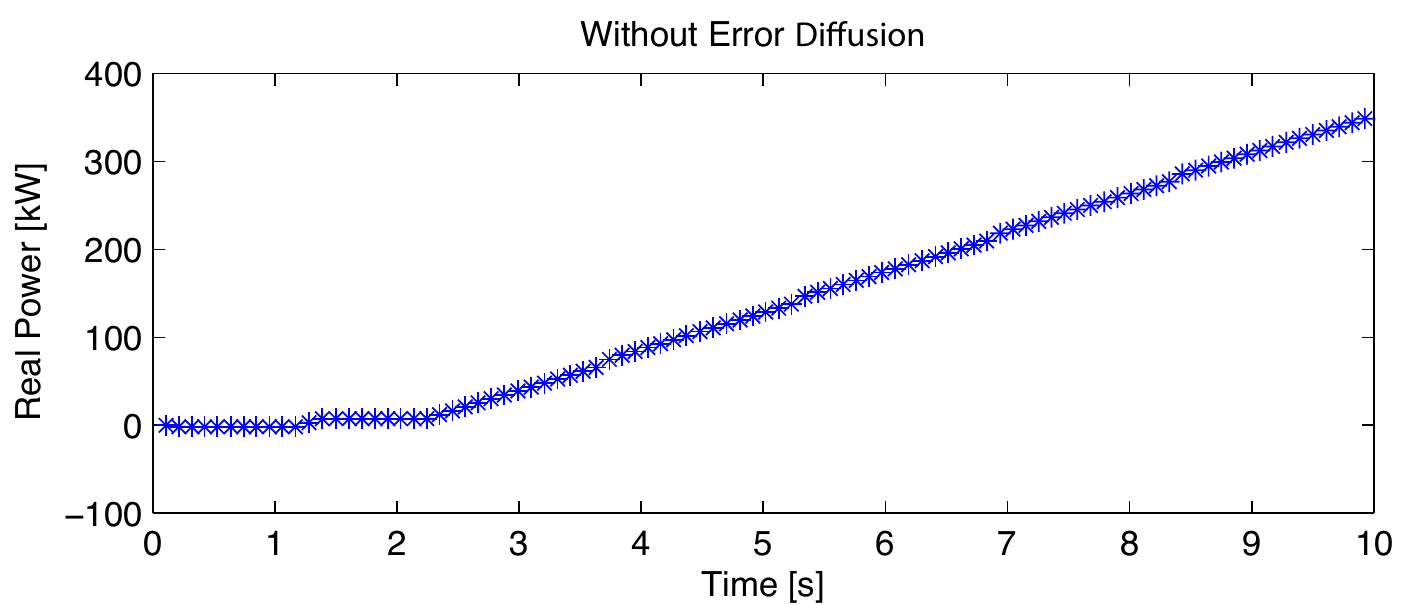}\\
\includegraphics[width=.8\columnwidth]{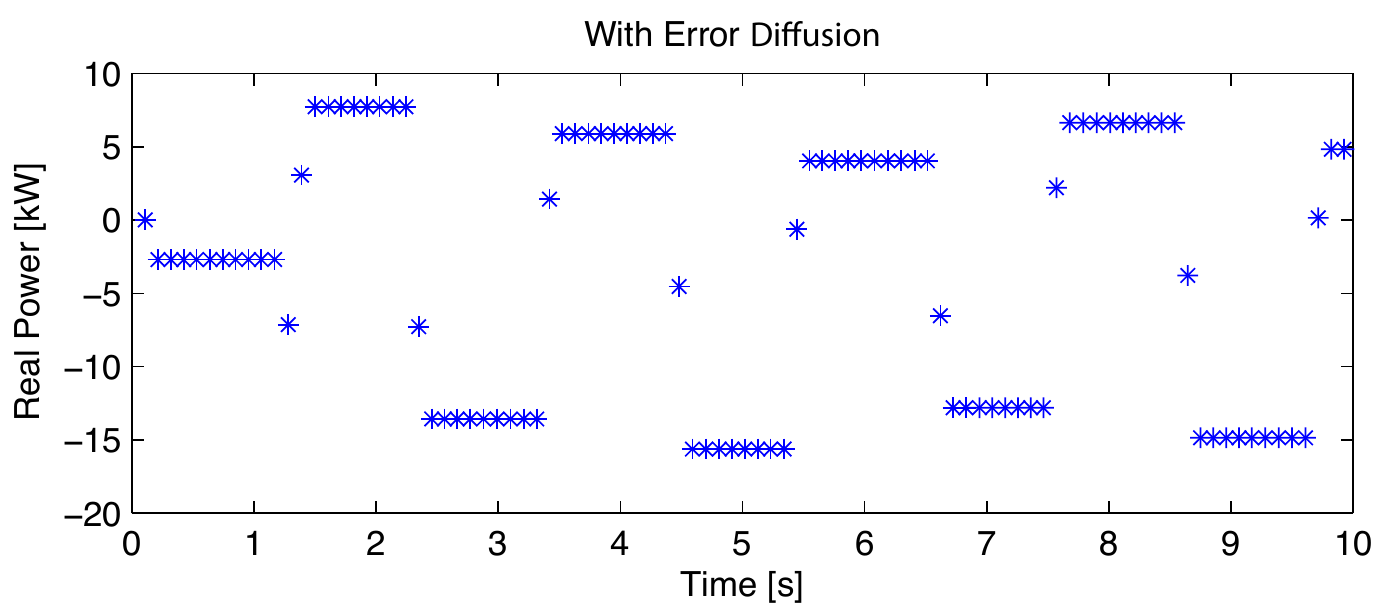}
\caption{Plots of the accumulated error of the heater agent. In the absence of error diffusion (top figure), the accumulated error grows linearly with time. With error diffusion (bottom figure), the accumulated error is bounded from above and below.}
\label{fig:accerror_heater}
\end{figure}

\begin{figure}
\centering
\includegraphics[width=.8\columnwidth]{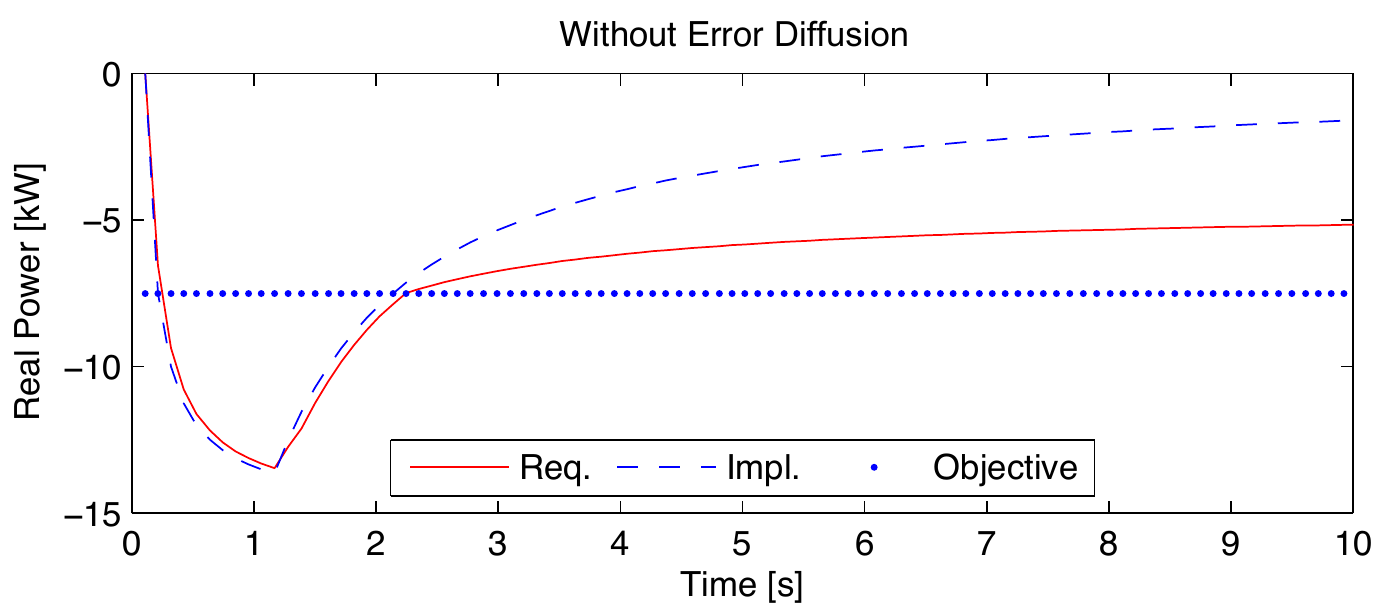}\\
\includegraphics[width=.8\columnwidth]{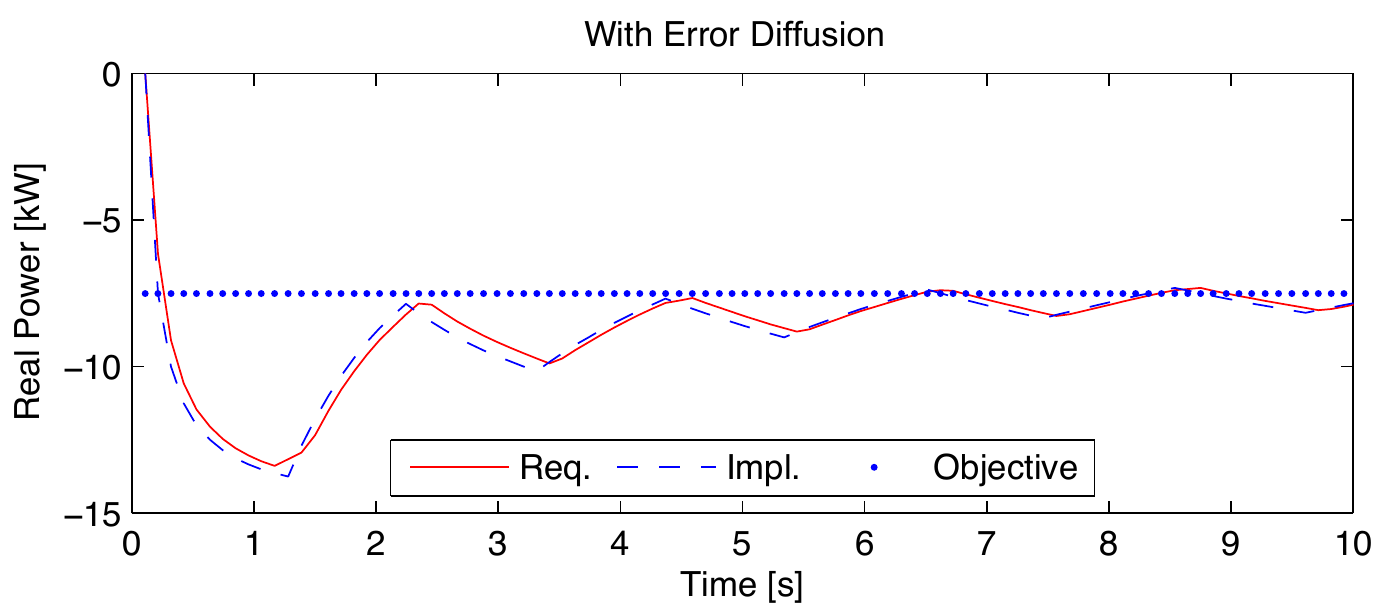}
\caption{Time-averaged sequences of requested and implemented setpoints belonging to the heater agent. In the absence of error diffusion (top figure), the implemented setpoint does not converge to the objective, where the objective is the minimizer of the cost function. With error diffusion (bottom figure), the sequence of implemented setpoint converges towards the objective.}
\label{fig:convergence}
\end{figure}

\begin{figure}
\centering
\includegraphics[width=.8\columnwidth]{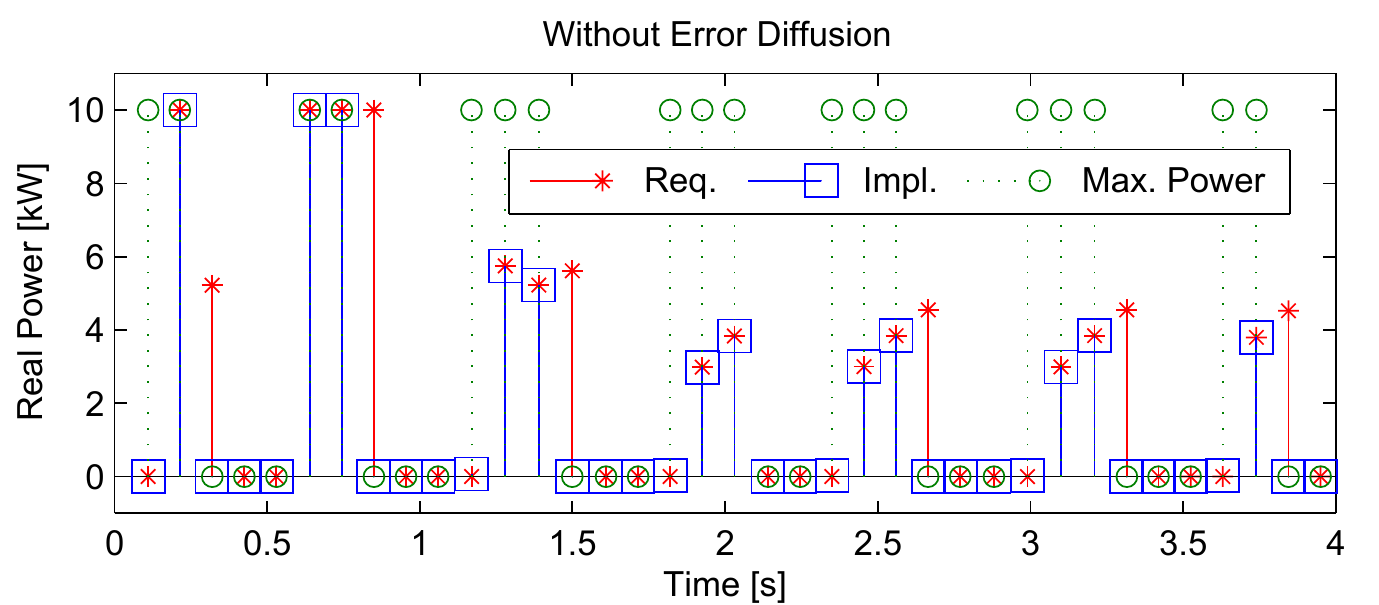}
\includegraphics[width=.8\columnwidth]{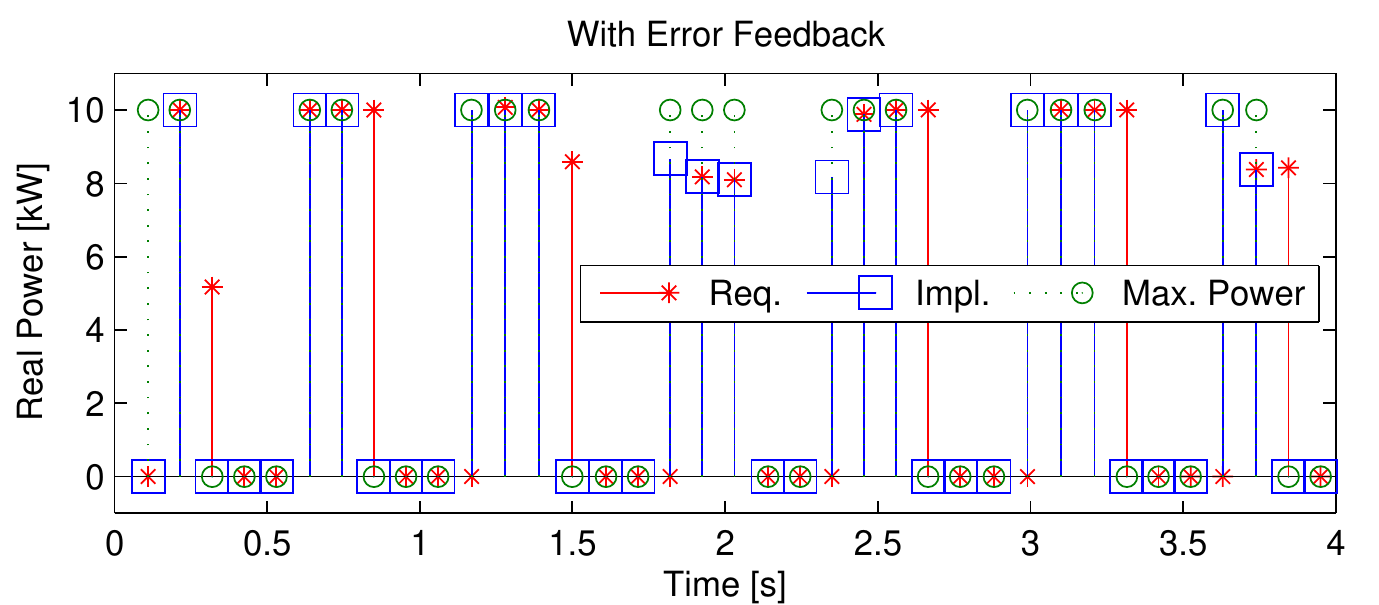}
\caption{Sequence of requested vs. implemented setpoints belonging to the PV agent. The bottom figure shows that error diffusion helps to maximise utilization of the PV.}
\label{fig:utilization}
\end{figure}

\begin{figure}
\centering
\includegraphics[width=.8\columnwidth]{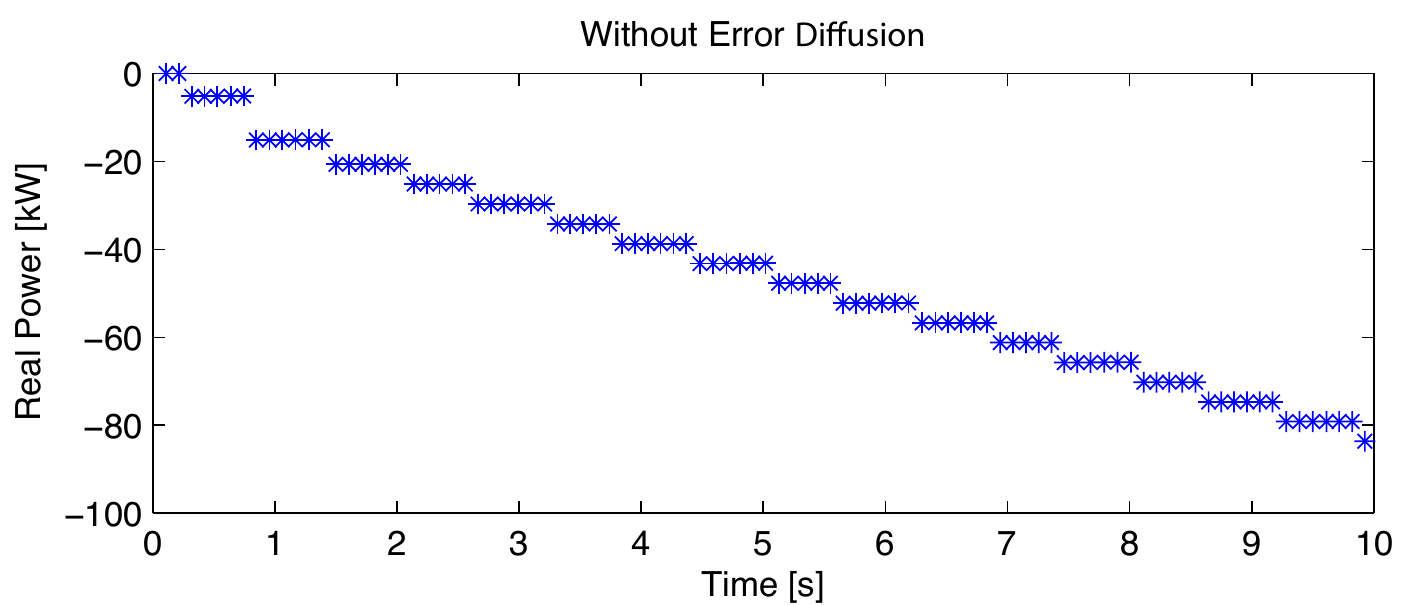}
\includegraphics[width=.8\columnwidth]{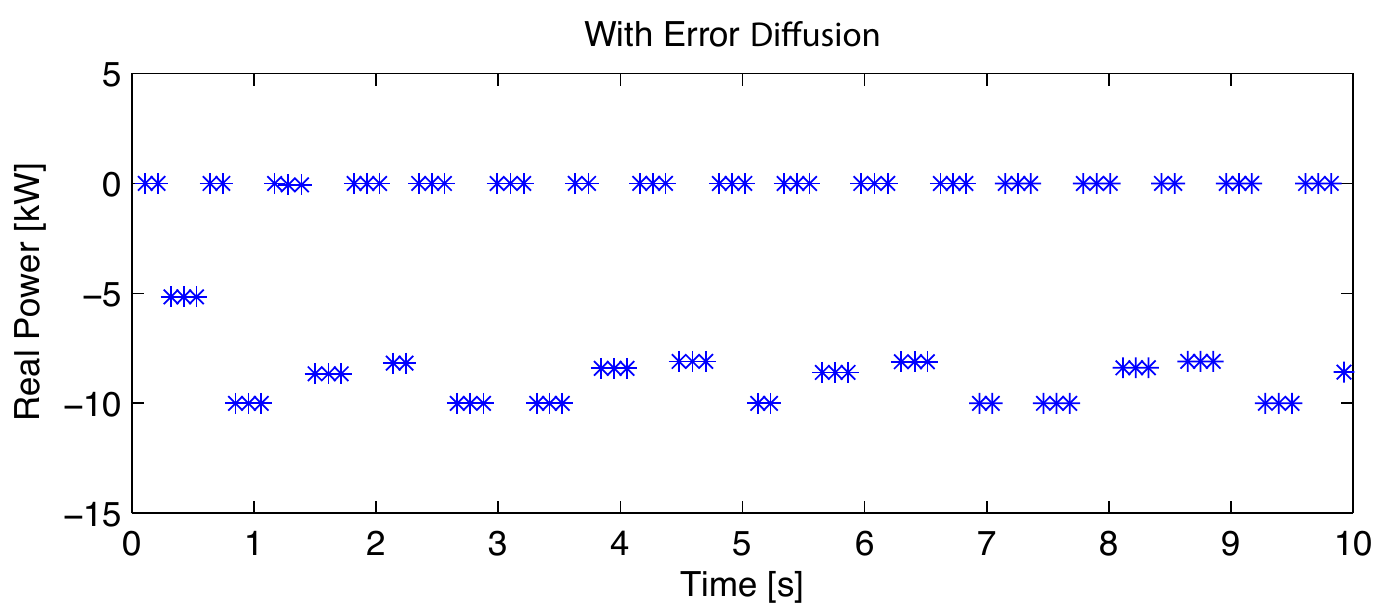}
\caption{Plot of the accumulated error of the PV agent. Like in the heater-agent case, the accumulated error grows unboundedly in the absence of error diffusion (top figure), whereas the accumulated error is bounded with error diffusion.}
\label{fig:accerror_PV}
\end{figure}

\begin{figure}
\centering
\includegraphics[width=.8\columnwidth]{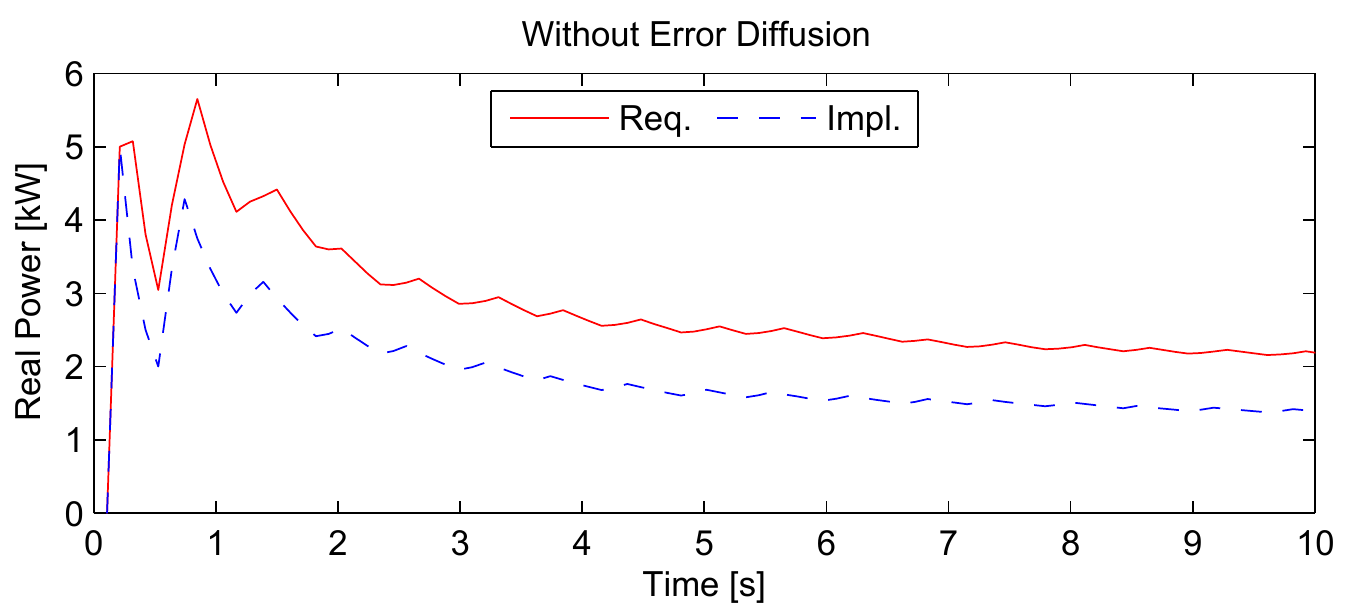}
\includegraphics[width=.8\columnwidth]{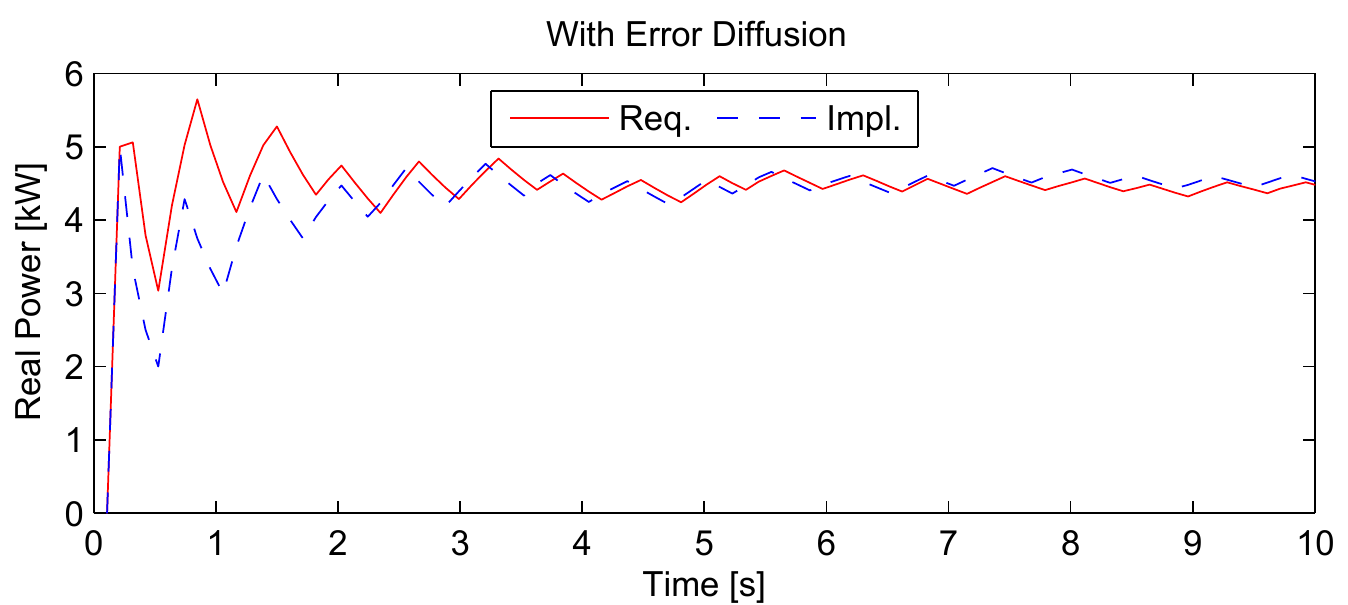}
\caption{Time-averaged sequence of requested vs. implemented setpoints belonging to the PV agent. Also this plot shows that error diffusion helps to increase the  utilization of the PV.}
\label{fig:errorDiffPV}
\end{figure}

\section{Conclusion} \label{sec:conc}
We have introduced a new property, bounded accumulated-error,
and we have shown that, if all resource agents have this property,
the performance of the overall grid-control system improves in several ways.
Hence, we conclude that every resource agent should
have $c$-bounded accumulated-error by design, for some appropriate (and scenario-specific) $c$.

We hope that this theoretically-oriented work
will contribute to the development of a practical and scalable solution
for controlling power grids with a significant fraction of renewable-energy sources.

While the motivation for this work originated from problems that are specific to our application, we believe that our results from \refsec{gen} are general enough to be of independent interest. 
\ifCLASSOPTIONcaptionsoff
  \newpage
\fi



\bibliographystyle{IEEEtran}
\bibliography{IEEEabrv,refs}
\end{document}